\documentclass{amsart}
\usepackage{latexsym,amsxtra,amscd,ifthen,amsmath,color,hyperref,float}
\usepackage{amsfonts}
\usepackage{verbatim}
\usepackage{amsmath}
\usepackage{amsthm}
\usepackage{amssymb}
\usepackage{fancybox}
\usepackage[notcite,notref,final]{showkeys}
 \usepackage[all,cmtip]{xy}
 \usepackage[normalem]{ulem}

\numberwithin{equation}{section}

\theoremstyle{plain}
\newtheorem{theorem}{Theorem}[section]
\newtheorem{lemma}[theorem]{Lemma}
\newtheorem{sublemma}[theorem]{Sublemma}
\newtheorem{proposition}[theorem]{Proposition}

\newtheorem{corollary}[theorem]{Corollary}

\theoremstyle{definition}
\newtheorem{definition}[theorem]{Definition}
\newtheorem{example}[theorem]{Example}

\newtheorem{hypothesis}[theorem]{Hypothesis}
\newtheorem{remark}[theorem]{Remark}

\newcommand{\mf}{\mathfrak}
\newcommand{\mc}{\mathcal}

\newcommand{\kk}{{\Bbbk}}

\newcommand{\ug}{\underline{g}}
\newcommand{\uchi}{\underline{\chi}}

\makeatletter              
\let\c@equation\c@theorem  
\makeatother

\begin{document}

\title[Pointed Hopf actions on fields, II]
{Pointed Hopf actions on fields, II}

\author{Pavel Etingof and Chelsea Walton}

\address{Etingof: Department of Mathematics, Massachusetts Institute of Technology, Cambridge, Massachusetts 02139,
USA}

\email{etingof@math.mit.edu}

\address{Walton: Department of Mathematics, Temple University, Philadelphia, Pennsylvania 19122,
USA}

\email{notlaw@temple.edu}

\bibliographystyle{abbrv}       

\begin{abstract}
This is a continuation of the authors' study of finite-dimensional pointed Hopf algebras $H$ which act  inner faithfully on commutative domains. As mentioned in Part I of this work, the study boils down to the case where $H$ acts inner faithfully on a field. These Hopf algebras are referred to as Galois-theoretical.

In this work, we provide classification results for finite-dimensional pointed Galois-theoretical Hopf algebras $H$ of finite Cartan type. Namely, we determine 
when such $H$ of type A$_1^{\times r}$ and some $H$ of rank two possess the Galois-theoretical property. 
Moreover, we provide necessary and sufficient conditions for Reshetikhin twists of small quantum groups to be Galois-theoretical.
\end{abstract}


\maketitle



\section{Introduction} \label{sec:intro}

The goal of this paper is to continue the study of finite-dimensional pointed Hopf actions on commutative domains over 
an algebraically closed field $\kk$ of characteristic zero, which we started in \cite{PartI}. 
\footnote{The reference numbers for the published version of \cite{PartI} differ from the reference numbers of the preprint version; see the Appendix of {\tt http://arxiv.org/abs/1403.4673}(v4) for the key.}
By \cite[Lemma~9 and Remark~3]{PartI}, this reduces to the study of Hopf actions on fields containing $\kk$. Moreover, by passing to appropriate Hopf quotients,
 it suffices to consider {\it inner faithful} actions, i.e., those not factoring through a `smaller' Hopf algebra (Definition~\ref{def:infaith}), and we will do so throughout the paper.
Since faithful actions of finite groups on fields are studied by Galois theory, in \cite{PartI} we made the following definition. 

\begin{definition} \label{def:GT}
A Hopf algebra $H$ over $\kk$ is said to be {\it Galois-theoretical} if it  acts inner faithfully on a field containing $\kk$.
\end{definition}

Examples of Galois-theoretical Hopf algebras are provided in \cite[Theorem~2]{PartI}. They include Taft algebras, $u_q(\mf{sl}_2)$, and twists of small quantum groups.
At the same time, it is shown in \cite{PartI} that many familiar examples of finite-dimensional Hopf algebras 
(such as ${\rm gr}(u_q(\mf{sl}_2))$ and generalized Taft algebras) are not Galois-theoretical. 

Our goal here is to classify Galois-theoretical finite-dimensional pointed Hopf algebras in as many cases as possible. We believe that the methods of this paper 
with some additional work can lead to a complete solution of this problem, at least in the special case when the group of grouplike elements 
is abelian.  

Recall that an important invariant of a pointed Hopf algebra is its {\it rank} $\theta$ (Definition~\ref{def:rank}). 
Our first result discusses the rank one case. Namely, let $H$ be a finite-dimensional pointed Hopf algebra of rank one. 
Then by \cite[Theorem~1(a)]{KropRadford}, $H$ is generated by the group of grouplike elements $G=G(H)$ and a $(g,1)$-skew-primitive element $x$, for some $g$ contained in the center 
$Z(G)$ of $G$. 

\begin{theorem}[Theorem~\ref{thm:rank1}] \label{thm:rank1intro}
 Let $H$ be a finite-dimensional pointed Hopf algebra of rank one, as above. Then, $H$ is Galois-theoretical if and only if the Hopf subalgebra generated by $\{g,x\}$ is a Taft algebra.
\end{theorem}

For finite-dimensional pointed Hopf algebras $H$ of higher rank, we make the following assumptions for the rest of the paper, unless stated otherwise.

\begin{hypothesis}\label{hyp:intro} Take $H$ to be a finite-dimensional pointed Hopf algebra and assume that $G=G(H)$ is abelian. Now the associated coradically graded Hopf algebra gr($H$) is isomorphic to the bosonization of a Nichols algebra $\mf{B}(V)$ by $\kk G$ (see Proposition~\ref{prop:Angiono}). In other words, $H$ is a lifting of $\mf{B}(V) \# \kk G$. Assume that $\mathfrak{B}(V)$ is of {\it finite Cartan type} (see Definition~\ref{def:V,c}(b)).  
\footnote{In fact, all finite-dimensional pointed Hopf algebras with $G$ abelian, and with all prime divisors of $|G|$ being $>7$ (subject to additional mild conditions), are liftings of bosonizations of Nichols algebras of finite Cartan type by $\kk G$ \cite{AndSch:pointed}. See Theorem~\ref{thm:finiteCartan} for more details.}
\end{hypothesis}

Under the hypotheses above, $H$ is generated by $G$ and $(g_i, 1)$-skew-primitive elements $x_i$ for $g_i \in G$, so that $g_i x_j = \chi_j(g_i) x_j g_i$, where $\chi_j$ is a character of $G$. 
If $x_i^{n_i}=0$ for some $n_i \geq 2$, then $H$ is a finite-dimensional Hopf quotient of the (typically 
infinite-dimensional) pointed Hopf algebra $H(G,g_1,\dots,g_\theta,\chi_1,\dots,\chi_\theta)$ of rank~$\theta$ from Definition~\ref{def:H(G,theta)}. 

For this reason, the following theorem will play a central role in this paper.
For any $g\in G$ and $\chi\in \widehat{G}$ (the character group of $G$), 
let $I_{g,\chi}:=\lbrace{1\le i\le \theta ~|~g_i=g,\chi_i=\chi\rbrace}$.

\begin{theorem}[Theorem~\ref{thm:H(G,theta)}] \label{thm:H(G,theta)intro}
Let $L$ be a field containing $\kk$ and equipped with a faithful action of $G$. 
Then, the extensions of the $G$-action on $L$ to a (not necessarily inner faithful) $H(G, g_1,\dots,g_\theta, \chi_1,\dots,\chi_\theta)$-action
are defined by the formula
\begin{equation} 
x_i \mapsto w_i(1-g_i),
\end{equation}
where $w_i \in L$ is such that  $g \cdot w_i = \chi_i(g) w_i$ for all $g \in G$. 

Moreover,  for each $g,\chi$, the skew-primitive elements $\lbrace{x_i ~|~ i\in I_{g,\chi}\rbrace}$ are linearly independent  as $\kk$-linear operators on $L$
if and only if so are the elements $w_i$; this can be achieved by an appropriate choice of  $w_i$.
\footnote{In many cases, $|I_{g,\chi}| \le 1$, and the condition that $\lbrace{x_i ~|~ i\in I_{g,\chi}\rbrace}$ are linearly independent boils down to the 
condition $x_i\ne 0$ as operators on $L$.}  
\end{theorem}

Using this result, we provide in Example~\ref{sl2ex}
an illustration of how to construct an inner faithful action of a finite-dimensional pointed Hopf algebra on a field as an extension of a faithful group action. 

Now we continue our study in the coradically graded case. We have the following result for type A$_1^{\times \theta}$.

\begin{theorem}[Theorem~\ref{thm:B(G,theta)}] \label{thm:B(G,theta)intro}
Let $H$ be a finite-dimensional pointed coradically graded  Hopf algebra of finite Cartan type A$_1^{\times \theta}$. Then, $H$ is Galois-theoretical if and only if the Hopf subalgebra $H'$ of $H$ generated  by $\{g_1, \dots, g_{\theta}, x_1, \dots, x_{\theta}\}$ is the tensor product of 
\begin{itemize}
\item Taft algebras $T(n,\zeta)$, 
\item Nichols Hopf algebras $E(n)$, and 
\item book algebras $\mathbf{h}(\zeta,1)$. 
\end{itemize}
\end{theorem}

\noindent See \cite[Sections 3.1, 3.2, 3.4]{PartI} for the presentations of $T(n,\zeta)$ (=:$T(n)$), $E(n)$, and $\mathbf{h}(\zeta,1)$, respectively.

The next theorem describes the Galois-theoretical properties of finite-dimensional pointed coradically graded  Hopf algebras of rank two.

\begin{theorem}[Theorem~\ref{thm:ranktwo}] \label{thm:ranktwointro}
Let $H$ be a finite-dimensional pointed coradically graded  Hopf algebra of rank two, subject to condition \eqref{eq:conditions} on the values of $\chi_i(g_i)$.
Then, $H$ is Galois-theoretical if and only if the Hopf subalgebra $H'$ generated by $g_1, g_2, x_1, x_2$ is one of the following:
\begin{enumerate}
\item of type A$_1 \times$A$_1$, namely
\begin{itemize}
\item the tensor product of Taft algebras  $T(n,\zeta) \otimes T(n',\zeta')$ for $n, n' \geq 2$, or
\item the 8-dimensional Nichols Hopf algebra $E(2)$, or 
\item the book algebra $\mathbf{h}(\zeta,1)$;
\end{itemize}
\item of type A$_2$, B$_2$, or G$_2$ with $\chi_2(g_1) = 1$ or $\chi_1(g_2)=1$ \textnormal{(}here, $H'$ is isomorphic to a twist $u_q^{\geq 0}(\mathfrak{g})^{J}$\textnormal{)}; or
\item one of the $3^4$-dimensional (resp. $5^5$-dimensional, $7^7$-dimensional)  Hopf algebras of type A$_2$ (resp. B$_2$, G$_2$) from \cite[Theorem~1.3(ii) (resp. (iii), (iv))]{AndSch:quantum}, where $\chi_2(g_1)$, $\chi_1(g_2) \neq 1$.
\end{enumerate}
\end{theorem}

Due to the complexity of the rank two case, we leave the general study of coradically graded Galois-theoretical Hopf algebras of higher rank to future work.

Next, we discuss the Galois-theoretical properties of liftings $H$ of coradically graded finite-dimensional pointed Hopf algebras 
in the cases when $H$ is of finite Cartan type A$_1^{\times \theta}$ and of rank two. 
To our knowledge, liftings of Nichols algebras of type G$_2$ 
are unclassified, so we do not address this case here.

\begin{theorem} [Theorem~\ref{thm:liftQLS}]  \label{thm:liftA1intro}
Let $H$ be a lifting of bosonization of a Nichols algebra $\mf{B}(V)$ of finite Cartan type as classified in \cite{AndSch:p3} for type A$_1^{\times \theta}$.
Then, $H$ is Galois-theoretical if and only if the Hopf subalgebra of $H$ generated by $\{g_1, \dots, g_{\theta},$ $ x_1, \dots, x_{\theta}\}$ is the quotient by a group of central grouplike elements of a tensor product of:
\begin{itemize}
\item $u_q'({\mathfrak{gl}}_2)$, the Hopf algebras from \cite[Definition~11]{PartI},
\item Taft algebras $T(n,\zeta)$,
\item Nichols Hopf algebras $E(n)$, and 
\item book algebras $\mathbf{h}(\zeta,1)$.
\end{itemize}
\end{theorem}

\begin{theorem} [Propositions~\ref{prop:liftA2},~\ref{prop:liftB2}]  \label{thm:liftA2B2intro}
Let $H$ be a lifting of bosonization of a Nichols algebra $\mf{B}(V)$ of finite Cartan type as classified  in \cite{AndSch:p4} for type A$_2$, and in \cite{BDR} for type B$_2$. Here, we assume that the braiding parameter is a root of unity of odd order $m >1$, with $m \geq 5$ for type A$_2$, and $m \neq 5$ for type B$_2$. In either case, if $H$ is Galois-theoretical, then $H$ is coradically graded. 
\end{theorem}

Finally, we address the Galois-theoretical property of twists of small quantum groups.

\begin{theorem}[Propositions~\ref{prop:uq(g)} and~\ref{prop:uq(g)twist}] \label{thm:uq(g)intro} 
Let $\mf{g}$ be a finite-dimensional semisimple Lie algebra with Cartan matrix $(a_{ij})$. Let $q \in \kk$ be a root of unity of odd order $m \geq 3$, with $m>3$ for type G$_2$.  For parts \textnormal{(}d\textnormal{)} and \textnormal{(}e\textnormal{)} below, we also require that $m$ is relatively prime to det$(a_{ij})$, and to 3 in type G$_2$. Let $J$ be a Drinfeld twist coming from the Cartan subgroup of the small quantum group
\textnormal{(}i.e., a Reshetikhin twist\textnormal{)}. Then:
\begin{enumerate}
\item $u_q(\mf{g})$ is Galois-theoretical if and only if $\mf{g} = \mf{sl}_2$.
\item $u_q^{\geq 0}(\mf{g})$ is Galois-theoretical if and only if $\mf{g} = \mf{sl}_2$.
\item gr$(u_q(\mf{g}))$ is not Galois-theoretical.
\item $u_q(\mf{g})^J$ can be Galois-theoretical if and only if $\mf{g} = \mf{sl}_n$, and in this case, there are only two of such twists $J$
for $n \geq 3$, and one \textnormal{(}namely, $J=1$\textnormal{)} for $n=2$.
\item There are precisely $2^{\text{rank}(\mf{g})-1}$ twists $J$ for which $u_q^{\geq 0}(\mf{g})^J$ is Galois-theoretical.
\end{enumerate}
\end{theorem}

In the theorem above, twists are counted up to gauge transformations, as gauge equivalent twists produce isomorphic Hopf algebras. 
\medskip

The paper is organized as follows. Background material on Hopf algebra actions and pointed Hopf algebras of finite Cartan type is provided in Section~\ref{sec:background}. Preliminary results on the Galois-theoretical property of finite-dimensional pointed Hopf algebras of finite Cartan type are provided in Section~\ref{sec:pointed}. 
Here, we prove Theorem~\ref{thm:H(G,theta)intro} and define  {\it minimal} Hopf algebras which will be used throughout this work. 
Next, in Section~\ref{sec:coradA1r}, we study the Galois-theoretical property in the type A$_1^{\times \theta}$ case, which includes the rank one case; we establish Theorems~\ref{thm:rank1intro} and~\ref{thm:B(G,theta)intro} here. In Section~\ref{sec:rank2}, we restrict our attention to the coradically graded case and study Galois-theoretical $H$ of rank two; namely, we prove Theorem~\ref{thm:ranktwointro}.
Then, in Section~\ref{sec:noncorad}, we determine when liftings of bosonizations of certain Nichols algebras of finite Cartan type are Galois-theoretical; we verify Theorems~\ref{thm:liftA1intro} and~\ref{thm:liftA2B2intro} here. Finally, in Section~\ref{sec:smallqgroups}, we prove Theorem~\ref{thm:uq(g)intro} on the Galois-theoretical property of twists of small quantum groups. Suggestions for further directions of this work are presented in Section~\ref{sec:directions}.


\section{Background material} \label{sec:background}

In this section, we provide a background discussion of Hopf algebra actions and pointed Hopf algebras of finite Cartan type. We set the following notation for the rest of the article. Unless specified otherwise: 
\medskip

\hspace{.1in} $\kk$ =  an algebraically closed base field of characteristic zero; all unadorned tensor products are over $\kk$\\
\indent $\zeta$, $q$  = a primitive root of unity in $\kk$ of order $n$ and $m$, respectively\\
\indent \hspace{.06in} $H$ = a  finite-dimensional Hopf algebra with coproduct $\Delta$, counit $\varepsilon$, antipode $S$\\
\indent \hspace{.06in} $G$ = the group of grouplike elements $G(H)$ of $H$\\
\indent \hspace{.06in} $\widehat{G}$ = character group of $G$ = $\{\alpha: G \rightarrow \kk^{\times}\}$\\
\indent \hspace{-.05in} $I_{g,\chi}$ = the index set $\{ i ~|~ g_i = g,~\chi_i=\chi \}$ for given elements $g \in G$, $\chi \in \widehat{G}$\\
\indent \hspace{.01in} $q_{ij}$ = the scalar $\chi_j(g_i)$ for $\chi_j \in \widehat{G}$ and $g_i \in G$\\
\indent \hspace{.03in} $n_i$ = the order of $q_{ii}$\\
\indent \hspace{-.12in} $(a_{ij})$ =  the Cartan matrix associated to a finite-dimensional semisimple Lie algebra\\
\indent \hspace{.06in} $L$ = an $H$-module field containing  $\kk$\\
\indent \hspace{.06in} $F$ = the subfield of invariants $L^H$\\
\indent \hspace{.09in}$E$ = an intermediate field of the extension $L^H \subset L$\\
\indent \hspace{-.1in}${}^G_G \mc{YD}$ = the category of Yetter-Drinfeld modules over $\kk G$\\
\indent \hspace{-.15in} $(V,c)$ = finite-dimensional braided vector space in ${}^G_G\mc{YD}$ with basis $x_1, \dots, x_{\theta}$
\smallskip

\noindent Since $G$ is abelian by Hypothesis~\ref{hyp:intro}, we have that $V$ is just a $G$-graded $G$-module.

\subsection{Hopf algebra actions} We recall basic facts about Hopf algebra actions; refer to
\cite{Montgomery} for further details.  
A left $H$-module $M$ has a left $H$-action structure map denoted by $\cdot : H \otimes M \rightarrow M$.

\begin{definition} \label{def:Hopfact} Given a Hopf algebra $H$ and an algebra $A$, we say
that {\it $H$ acts on $A$} (from the left) if $A$ is a left $H$-module,
$h \cdot (ab) = \sum (h_1 \cdot a)(h_2 \cdot b)$,
and $h \cdot 1_A = \varepsilon(h)
1_A$ for all $h \in H$,  $a,b \in A$. Here, $\Delta(h) =
\sum h_1 \otimes h_2$ (Sweedler notation). In the case that $H$ acts on a field $L$, we refer to $L$ as an {\it $H$-module field}.
\end{definition}

We restrict ourselves to $H$-actions that do not factor
through `smaller' Hopf algebras.

\begin{definition} \label{def:infaith} Given a left $H$-module $M$, we say that $M$ is an {\it
inner faithful} $H$-module if $IM\neq 0$ for every nonzero Hopf ideal
$I$ of $H$. Given an action of a Hopf algebra $H$ on an algebra $A$, we say that this
 action is  {\it inner faithful} if the left $H$-module (algebra) $A$ is inner
faithful.
\end{definition}

When given an $H$-action on an algebra $A$, one can always pass uniquely to an inner faithful $\overline{H}$-action on $A$, where $\overline{H}$ is some quotient Hopf algebra of $H$. 

Next, we  consider elements of a field  $L$  invariant under the $H$-action on $L$.

\begin{definition}
Let $H$ be a Hopf algebra that acts on a field $L$ containing $\kk$ from the left. The {\it subfield of invariants} for this action is given by
$$L^H = \{ \ell \in L ~|~ h \cdot \ell  = \varepsilon(h) \ell~~ \text{ for all } h \in H\}.$$
\end{definition}

\subsection{Pointed Hopf algebras of finite Cartan type} \label{sec:finiteCartan}
Let us begin with a discussion of pointed Hopf algebras. A nonzero element $g \in H$ is {\it grouplike} if $\Delta(g) = g \otimes g$, and the group of grouplike elements of $H$ is denoted by $G=G(H)$. An element $x \in H$ is 
{\it $(g,g')$-skew-primitive}, if for grouplike elements $g, g'$ of $G(H)$, we have that $\Delta(x) = g \otimes x + x \otimes g'$. The space of such elements is denoted by $P_{g,g'}(H)$. The {\it coradical} $H_0$ of  a Hopf algebra $H$ is the sum of all simple subcoalgebras of $H$. The {\it coradical filtration} $\{H_n\}_{n \geq 0}$ of $H$ is defined inductively by 
$$H_n = \Delta^{-1}(H \otimes H_{n-1} + H_0 \otimes H),  \quad \text{for } n \geq 1$$
where $H= \bigcup_{n \geq 0} H_n$.
We say that a Hopf algebra $H$ is {\it pointed}  if all of its simple $H$-comodules (or equivalently, if all of its simple $H$-subcoalgebras) are 1-dimensional. When $H$ is pointed, we have that $H_0 = \kk  G$ and $H_1=\kk G + \sum_{g,g'\in G}P_{g,g'}(H)$. 
Following \cite{KropRadford}, we measure the complexity of a pointed Hopf algebra by considering its {\it rank}.

\begin{definition} \label{def:rank}
Let $H$ be a  Hopf algebra  with coradical filtration $\{H_n\}_{n \geq 0}$, where the coradical $H_0$ is a Hopf subalgebra and $H$ is generated by $H_1$ as an algebra. The {\it rank} of $H$ is $\theta$ if $\dim (\kk \otimes_{H_0} H_1) = \theta+1$.
\end{definition}

To study finite-dimensional pointed Hopf algebras $H$, it is convenient to assume that $G=G(H)$ is abelian (as we have done in Hypothesis~\ref{hyp:intro}) since we have the following result.

\begin{proposition} \label{prop:Angiono} \cite{Angiono}
Let $H$ be a finite-dimensional pointed Hopf algebra with an abelian group of grouplike elements $G$. Then, 
the associated coradically graded Hopf algebra $gr(H)$ is  isomorphic to a bosonization $\mf{B}(V) \# \kk G$ of a Nichols algebra $\mf{B}(V)$ by the group algebra $\kk G$. In other words, $H$ is generated by grouplike  and skew-primitive elements.
\end{proposition}

\begin{proof} 
This follows from \cite[Theorem~4.15]{Angiono}. Namely, such an $H$ satisfies the Andruskiewitsch-Schneider conjecture \cite[Conjecture~1.4]{AndSch:finite}: it is generated in degree one (by grouplike and skew-primitive elements). 
\end{proof}

We refer the reader to \cite[Section~2]{AndSch:survey} for a review of Nichols algebras associated to braided vector spaces.

Now we consider a special subclass of the finite-dimensional pointed Hopf algebras, those of {\it finite Cartan type}. Refer to \cite[Section~1.2]{AndSch:survey} and \cite{AndSch:pointed} for further details.

\begin{definition}{} \label{def:V,c} 
Let $(V,c)$ be a finite-dimensional braided vector space. 
\begin{enumerate}
\item $(V,c)$ is of {\it diagonal type} if there exists a basis $x_1, \dots, x_{\theta}$ of $V$ and scalars $q_{ij}\in \kk^\times$ so that
$$c(x_i \otimes x_j) = q_{ij} x_j \otimes x_i$$
for all $1 \leq i,j \leq \theta$. The matrix $(q_{ij})$ is called the {\it braiding matrix}.
\item $(V,c)$ is of {\it finite Cartan type} if it is of diagonal type and 
\begin{equation} \label{eq:qij Cartan}
q_{ii} \neq 1 \quad \text{ and } \quad q_{ij}q_{ji} = q_{ii}^{a_{ij}}
\end{equation}
where $(a_{ij})_{1 \leq i,j \leq \theta}$ is a Cartan matrix associated to a finite-dimensional semisimple Lie algebra.
\item The same terminology applies to the Nichols algebra $\mf{B}(V)$ and a Hopf algebra $H$ when $gr(H) \cong \mf{B}(V) \#\kk G$. In this case, $H$ is a {\it lifting} of a finite-dimensional pointed Hopf algebra of finite Cartan type, and further,  {\it trivial lifting} when $H$ is coradically graded.
\end{enumerate}
\end{definition}

Note that since $G(H)$ is abelian here, the corresponding braided vector space $(V,c)$ is automatically of diagonal type;
indeed, it suffices to choose $x_i$ to be eigenvectors of the $G$-action. 
See \cite[Theorem~2]{PartI} for examples of finite-dimensional pointed Hopf algebras of finite Cartan type.

Motivated by Proposition~\ref{prop:Angiono},  we now give the precise presentation of finite-dimensional pointed Hopf algebras $H$ where the order of $G(H)$ has large prime divisors; $H$ is of finite Cartan type in this case.
\medskip

\noindent {\it Notation} [$\Phi^+$, $\alpha$, $n_I$]. In the case where $V$ is of finite Cartan type, let $\Phi$ be the root system of the Cartan matrix $(a_{ij})_{1 \leq i,j, \leq \theta}$, and let $\Phi^+$ be the subset of positive roots.  Let $\alpha_1, \dots, \alpha_{\theta}$ be the simple roots. We write $i \sim j$ if the roots $\alpha_i$ and $\alpha_j$ are in the same connected component of the Dynkin diagram of $\Phi$. Let $\mc{X}$ be the set of such connected components. Assume that
\begin{equation} \label{eq:conditions}
\text{$q_{ii}$  has odd order, and is prime to 3 if $i$ lies in a component G$_2$.}
\end{equation}
The order of $q_{ii}$ and of $q_{jj}$ are equal when $i \sim j$. So, we set $n_I$ to be the order of any $q_{ii}$ with $\alpha_i \in I$ of $\mc{X}$.
\medskip

\begin{theorem} \label{thm:finiteCartan}
\textnormal{(a)}  \cite[Theorem~0.1(2)]{AndSch:pointed} Let $H$ be a finite-dimensional pointed Hopf algebra, where $G$ is abelian so that the prime divisors of $|G|$ are $>7$. Then, $H$ is of finite Cartan type.   

\textnormal{(b)}  \cite[Theorem~4.5]{AndSch:finite} Let $H$ be a finite-dimensional pointed Hopf algebra of finite Cartan type, with $G$  abelian.  Let $(q_{ij} := \chi_j(g_i))$ be the braiding matrix of $H$ and assume \eqref{eq:conditions}.
Then, $H$ is generated by $G$ and by $(g_i, 1)$-skew-primitive elements $x_i$ for $g_i \in G$ and $i = 1, \dots, \theta$. Further, $gr(H) \cong \mf{B}(V) \# \kk G$ is subject to the relations of $G$ along with the following relations:

\begin{center}
\begin{tabular}{rl}
$g x_i = \chi_i(g)x_i g$ & for all $i$, and all $g \in G$\\
ad$_c(x_i)^{1-a_{ij}}(x_j) = 0$ & for all $i \neq j$\\
$x_{\alpha}^{n_I} = 0$ & for all $\alpha \in \Phi_I^+$, $I \in \mc{X}$.
\end{tabular}
\end{center}
\noindent Here, $(\text{ad}_c x_i)(y) = x_i y - (q_{ij_1} \cdots q_{ij_t}) y x_i$
for $y = x_{j_1} \cdots x_{j_t}$. 
\qed
\end{theorem}


\section{Preliminary results on pointed Galois-theoretical Hopf algebras} \label{sec:pointed}

We begin this section by examining actions of Hopf algebras $H(G,g_1,\dots,g_\theta,\chi_1,\dots,\chi_\theta)$, which are infinite-dimensional when $\theta \geq 2$. We also define and discuss  {\it minimal} Hopf algebras, which will be used throughout the rest of the paper.


\subsection{On the Hopf algebras $H(G,g_1,\dots,g_\theta,\chi_1,\dots,\chi_\theta)$} \label{subsec:H(G,theta)}
We consider the actions of the Hopf algebra $H(G,g_1,\dots,g_\theta,\chi_1,\dots,\chi_\theta)$ defined below.

\begin{definition} \label{def:H(G,theta)} Let $G$ be a finite abelian group. Let $g_i\in G$ be an element of order $n_i\ge 2$, for $i=1,\dots,\theta$.
Fix characters $\chi_i: G \rightarrow \kk^{\times}$, for $i = 1, \dots, \theta$, so that $q_{ii} := \chi_i(g_i)$ has order $n_i$. 
Let $\ug:=(g_1,\dots,g_\theta)$ and $\uchi:=(\chi_1,\dots,\chi_\theta)$. The Hopf algebra 
$H(G, \ug, \uchi)$
 is generated by  $G$ and $(g_i,1)$-skew-primitive elements $x_i$, for $i=1, \dots, \theta$, subject to the relations of $G$, 
$$x_i^{n_i} = 0  \quad \text{ and } \quad gx_i = \chi_i(g) x_i g,$$
for all $g \in G$.
\end{definition}

Thus, $H(G, \ug, \uchi)$ is a quotient of the bosonization of  $\kk\langle x_1, \dots, x_{\theta} \rangle$ by $\kk G$, by the Hopf ideal of relations $(x_i^{n_i})_{i=1}^{\theta}$. If $\theta \geq 2$, then $H(G, \ug, \uchi)$ is infinite-dimensional.
For a given $i$, the subalgebra of $H(G, \ug, \uchi)$ generated by $\{g_i, x_i\}$ is a Taft algebra.

\begin{proposition}\label{anypoi} Let $H$ be a finite-dimensional pointed Hopf algebra of rank $\theta$ 
with $G=G(H)$ abelian, generated by $G$ and $(g_i,1)$-skew-primitive elements 
$x_i$ for $i=1, \dots, \theta$. Then, any inner faithful action of $H$ on a field $L\supset \kk$ descends 
from an action of $H(G, \ug, \uchi)$ on $L$ \textnormal{(}via a surjective Hopf algebra homomorphism $H(G,\ug,\uchi)\to H$\textnormal{)}
so that,  for any $g,\chi$, $\lbrace{x_i~|~i\in I_{g,\chi}\rbrace}$ are linearly independent 
as $\kk$-linear operators on $L$. 
\end{proposition} 

\begin{proof}
Suppose $H$ is Galois-theoretical, and we are given an inner faithful action 
of $H$ on a field $L\supset \kk$. For each $1\le i\le \theta$, consider the Hopf subalgebra $H_i\subset H$ 
generated by $g_i$ and $x_i$. By \cite[Theorem~1(a)]{KropRadford}, $H_i$ is a generalized Taft algebra $T(n_i, m_i,\alpha_i)$, where $m_i$ divides $n_i$. 
Since $H_i$ is Galois-theoretical, by \cite[Propositions~10(3) and~21]{PartI}  we must have that
$H_i$ is an ordinary Taft algebra $T(n_i,\zeta_i)$, that is, $x_i^{n_i}=0$ and $g_i x_i = q_{ii} x_i g_i$ for some $q_{ii} \in \kk^{\times}$ with ord($q_{ii}$)$=n_i$. By taking $\chi_i \in \widehat{G}$ with $\chi_i(g_i) = q_{ii}$, this implies that the $H$-action on $L$ descends from an $H(G,\ug,\uchi)$-action.

Moreover, since $H$ acts inner faithfully on $L$, the operators defined by $x_i$ on $L$ must be linearly independent.  Indeed, if $\{x_i ~|~ i \in I_{g,\chi}\}$ are linearly dependent, then there exists an element $f = \sum_i a_i x_i$, for $a_i \in \kk$ not all zero, that acts  on $L$ as zero. So, $\langle f \rangle$ forms a Hopf ideal (as $f$ is $(g,1)$-skew-primitive), which contradicts inner faithfulness. 
\end{proof} 

Proposition \ref{anypoi} shows that it is important to determine the structure of the $H(G, \ug, \uchi)$-module fields $L$.
This is done in the following theorem. 

\begin{theorem} \label{thm:H(G,theta)}
Let $L$ be a field containing $\kk$ and equipped with a faithful action of $G$. 
Then, the extensions of the $G$-action on $L$ to a \textnormal{(}not necessarily inner faithful\textnormal{)} $H(G, \ug, \uchi)$-action
are defined by the formula
\begin{equation} \label{eq:formula xi}
x_i \mapsto w_i(1-g_i),
\end{equation}
where $w_i \in L$ is such that  $g \cdot w_i = \chi_i(g) w_i$ for all $g \in G$. In other words, $x_i$ acts as $w_i(1-g_i)$, as $\kk$-linear operators on $L$.

In this case:
\begin{enumerate}
\item $L$ is a Galois extension of the field of invariants $F = L^H=L^G$; and
\item $L$ is also a Galois extension of the intermediate field $E = L^{H'}=L^{G'}$, where $G'$ is the subgroup of $G$ generated by $g_1, \dots, g_{\theta}$, and $H'$ is the Hopf subalgebra of $H$ generated by $G'$ and $\{x_i\}_{i=1}^{\theta}$.
\end{enumerate}

Moreover,  for each $g,\chi$, the skew-primitive elements $\lbrace{x_i ~|~ i\in I_{g,\chi}\rbrace}$ are linearly independent  as $\kk$-linear operators on $L$
if and only if so are the elements $w_i$; this case can be achieved by an appropriate choice of elements $w_i \in L$. 
\end{theorem}

\begin{remark} \label{rem:w unique}
For any $\chi:G \rightarrow \kk^{\times}$, there exists $w \in L^{\times}$ such that  $g \cdot w = \chi(g) w$ for all $g \in G$. The element $w$ is unique up to multiplying by an element of $F:=L^G$. This follows by the Normal Basis Theorem: $L$ is a free $FG$-module of rank one.
As a consequence, the extension of the $G$-action to an $H(G, \ug, \uchi)$-action as in Theorem~\ref{thm:H(G,theta)} is unique up to 
renormalization, $x_i \mapsto \lambda_ix_i$ for $\lambda_i \in F$. 
\end{remark}

\smallskip

\noindent {\it Proof of Theorem~\ref{thm:H(G,theta)}.}
Take $F:= L^G$. Since $L$ is a free $FG$-module of rank one by the Normal Basis Theorem, for each $\alpha \in \widehat{G}$, we can choose $u_{\alpha} \in L^{\times}$ such that $g \cdot u_{\alpha} = \alpha(g)u_{\alpha}$; see Remark~\ref{rem:w unique}. Here,  $u_{\alpha} u_{\beta} = \psi(\alpha, \beta) u_{\alpha \beta}$, where $\psi$ is a 2-cocycle of $\widehat{G}$ with values in $F^\times$; that is to say, $\psi(\alpha, \beta) \psi(\alpha \beta, \gamma) = \psi(\beta, \gamma) \psi(\alpha, \beta \gamma)$ for all $\alpha, \beta, \gamma \in \widehat{G}$. This follows from the associativity of $L$: $(u_{\alpha} u_{\beta}) u_{\gamma} = u_{\alpha}(u_{\beta} u_{\gamma})$. 

 Since $g x_i = \chi_i(g) x_i g$ for all $g \in G$, we claim for all $\alpha \in \widehat{G}$ that 
\begin{equation} \label{eq:ci}
x_i \cdot u_{\alpha} = c_i(\alpha) u_{\alpha \chi_i},
\end{equation}
where $c_i(\alpha) \in F$ satisfies two conditions:
\begin{equation} \label{eq:ci1}
 c_i(\alpha) = 0, \text{\hspace{.2in} if } \alpha(g_i) =1;
\end{equation}
\begin{equation} \label{eq:ci2}
\alpha(g_i) \psi(\alpha, \beta \chi_i)c_i(\beta)  + \psi(\alpha \chi_i, \beta) c_i(\alpha)= \psi(\alpha, \beta) c_i(\alpha \beta).
\end{equation}

To verify \eqref{eq:ci1}, note that if $\alpha(g_i) =1$, then $u_{\alpha} \in L^{g_i}$. Since $x_i^{n_i}=0$, we can employ \cite[Theorem~11(i)]{PartI} to get $x_i \cdot u_{\alpha} = 0$. So, we conclude that $c_i(\alpha) =0$. To verify \eqref{eq:ci2}, compare the coefficients of $u_{\alpha \beta \chi_i}$ in the equation $x_i \cdot (u_{\alpha} u_{\beta}) = x_i \cdot (\psi(\alpha, \beta) u_{\alpha \beta})$ (using the coproduct of $x_i$ on the left hand side).

Let $K_i:=c_i(\chi_i)$. Then, setting $\alpha = \chi_i$ and $\beta= \chi_i^{m-1}$, we get from \eqref{eq:ci2}:
$$K_i \psi(\chi_i^2, \chi_i^{m-1}) + q_{ii} \psi(\chi_i, \chi_i^m) c_i(\chi_i^{m-1}) =
\psi(\chi_i, \chi_i^{m-1}) c_i(\chi_i^m).$$
Thus, setting $b_i(m) = \displaystyle \frac{c_i(\chi_i^m)}{q_{ii}^m \psi(\chi_i, \chi_i^m)}$, we get 
$$\displaystyle\frac{K_i ~\psi(\chi_i^2, \chi_i^{m-1})}{q_{ii}^m ~\psi(\chi_i, \chi_i^m) \psi(\chi_i, \chi_i^{m-1})} + b_i(m-1) = b_i(m).$$
Using the 2-cocycle property of $\psi$, we get that
$\displaystyle\frac{K_i}{q_{ii}^m ~\psi(\chi_i, \chi_i)} + b_i(m-1) = b_i(m).$
Since $b_i(0) =0$ by \eqref{eq:ci1}, we get 
$b_i(m) = \displaystyle\frac{K_i}{\psi(\chi_i, \chi_i)} (q_{ii}^{-1} + q_{ii}^{-2} + \dots + q_{ii}^{-m}).$
Hence
\begin{equation} \label{eq:ci3}
c_i(\chi_i^m) ~=~ K_i(1 + q_{ii} + \dots + q_{ii}^{m-1}) \frac{\psi(\chi_i, \chi_i^m)}{\psi(\chi_i,\chi_i)} ~=~ \widehat{K_i} \psi(\chi_i, \chi_i^m)(1-q_{ii}^m),
\end{equation}
where $\widehat{K_i} = \displaystyle \frac{K_i}{\psi(\chi_i, \chi_i) (1- q_{ii})}$. 

Now let $\beta \in \widehat{G}$ be a function so that $\beta(g_i) = 1$, and set $\alpha = \chi_i^m$ in \eqref{eq:ci2} to get
\begin{equation} \label{eq:ci4}
\psi(\chi_i^{m+1}, \beta)c_i(\chi_i^m) ~=~ \psi(\chi_i^m, \beta)c_i(\chi_i^m\beta).
\end{equation}
We get by \eqref{eq:ci3} and \eqref{eq:ci4} that
$$\psi(\chi_i^{m+1}, \beta) \widehat{K_i} \psi(\chi_i, \chi_i^m) (1-q_{ii}^m) = \psi(\chi_i^m, \beta) c_i (\chi_i^m \beta).$$
Hence 
$$c_i(\chi_i^m \beta) ~=~ \widehat{K_i}(1-q_{ii}^m) \frac{\psi(\chi_i  \chi_i^m, \beta) \psi(\chi_i, \chi_i^m)}{\psi(\chi_i^m, \beta)} ~=~ \widehat{K_i} (1-q_{ii}^m) \psi(\chi_i, \chi_i^m \beta).$$
Since any $\alpha \in \widehat{G}$ is of the form $\chi_i^m\beta$, where $\beta(g_i)=1$, and $0\leq m \leq$ ord($g_i$)$-1$, we have that
\begin{equation} \label{eq:ci5}
c_i(\alpha) = \widehat{K_i} (1-\alpha(g_i))\psi(\chi_i, \alpha).
\end{equation}
By setting $w_i:= \widehat{K_i} u_{\chi_i}$, we get
$$ w_i (1-g_i) \cdot u_{\alpha}  ~=~ \widehat{K_i} u_{\chi_i} \left(1 - \alpha(g_i)\right) u_{\alpha} 
~=~ \widehat{K_i} (1-\alpha(g_i)) \psi(\chi_i, \alpha) u_{\alpha \chi_i}.$$
Thus by \eqref{eq:ci} and \eqref{eq:ci5}, we get that $x_i \mapsto w_i(1-g_i)$. 
Moreover, $g \cdot w_i = \chi_i(g) w_i$ for all $g \in G$, as required. 

Conversely, if we choose $w_i \in L$  such that $g \cdot w_i = \chi(g) w_i$ for $g \in G$, then it is easy to see  that the formula $x_i \mapsto w_i(1-g_i)$ defines an extension of the $G$-action to an $H(G, \ug, \uchi)$-action.

Next, it is clear that the elements $\lbrace{x_i ~|~ i\in I_{g,\chi}\rbrace}$ are linearly independent  as $\kk$-linear operators on $L$ for each $g,\chi$
if and only if so are the elements $w_i$. The latter can be achieved since $\dim_\kk F=\infty$.  

Now we establish statements (a) and (b) as follows. By \cite[Corollary~13]{PartI}, we have that $L^H=L^G (=:F)$ and the extension $L^G \subset L$ is Galois.
For the same reasons, we also have that $L$ is a Galois extension of $E := L^{H'} = L^{G'}$.
\qed

\begin{remark} \label{rem:w_i} In \eqref{eq:formula xi}, we have $w_i=\mu_iu_{\chi_i}$, where $\mu_i\in F$, $u_{\chi_i} \in L$, so that $g \cdot \mu_i = \mu_i$ and $g \cdot u_{\chi_i} = \chi_i(g)u_{\chi_i}$. In several computations below, we take $\mu_i =1$ for all $i$ without loss of generality, since $g$ commutes with $\mu_i$.
\end{remark}

\begin{corollary} \label{cor:GnotAbel}
Let $G$ be a finite group, not necessarily abelian, and consider the Hopf algebra $H(G, \ug, \uchi)$:= $H(G, g_1, \dots, g_{\theta}, \chi_1, \dots, 
\chi_{\theta})$ defined as above, with $g_i$ belonging to the center $Z(G)$ of $G$. Then, extensions of a faithful $G$-action on a field $L\supset \kk$ to a
(not necessarily inner faithful) $H(G, \ug, \uchi)$-action are given by formula ~\eqref{eq:formula xi} as in Theorem~\ref{thm:H(G,theta)}.

Moreover,  for each $g,\chi$, the $(g,1)$-skew-primitive elements $\lbrace{x_i ~|~ i\in I_{g,\chi}\rbrace}$ are linearly independent  as $\kk$-linear operators on $L$
if and only if so are the elements $w_i$, and this can be achieved by an appropriate choice of the $w_i$. 
\end{corollary}

\begin{proof} 
Adapt the proof of Theorem~\ref{thm:H(G,theta)}, except  that $F=L^{Z(G)}$, not $L^G$, and $G/Z(G)$ may act nontrivially on $F$. 
So, $\{w_i\}$ are unique not up to elements of $F$, but up to elements of $L^G=F^{G/Z(G)}$.
\end{proof}

We now provide an example of how to construct an inner faithful action of a finite-dimensional pointed Hopf algebra on a field as an extension of a faithful group action.

\begin{example}\label{sl2ex}
Let $m \geq 2$ and let $q$ be a root of unity in $\kk$ with ord($q^2$)=$m$. Consider the small quantum group $u_q(\mf{sl}_2)$ generated by a grouplike element $k$, a ($k, 1$)-skew-primitive element $e$, and a ($1, k^{-1}$)-skew-primitive element $f$, subject to relations:
$$ ke = q^2ek,  \quad kf = q^{-2}fk, \quad e^m =0, \quad f^m = 0, \quad k^m = 1, \quad ef -  fe = \frac{k - k^{-1}}{q - q^{-1}}.$$
Here, $G(u_q(\mf{sl}_2)) = \mathbb{Z}_m$ and the Hopf algebra $H(G, \ug, \uchi)$ is generated by $k,e,f$, subject to the first five relations. To show that the action $H(G, \ug, \uchi)$ on a field descends to an action of $u_q(\mf{sl}_2)$, 
we only need to work with the last relation of $u_q(\mf{sl}_2)$.

We construct an action of $u_q(\mf{sl}_2)$  on $\kk(z)$ as follows. Let $G = \mathbb{Z}_m = \langle k \rangle$ act on $\kk(z)$ by $ k \cdot z = q^{-2}z$. 
Take $x_1 = e$ and $x_2 = kf$; these elements are $(k,1)$-skew-primitive. Then $\chi_1(k) = q^2$ and $\chi_2(k) = q^{-2}$, as $ke = q^2ek$ and $k(kf) = q^{-2}(kf)k$. The last relation of $u_q(\mf{sl}_2)$ then has the form
\begin{equation} \label{eq:blah}
q^2 x_1 x_2 - x_2 x_1 = (k^2 - 1)(q-q^{-1})^{-1}.
\end{equation} 
Thus, by Theorem~\ref{thm:H(G,theta)}, we get that
the condition for the $H(G, \ug, \uchi)$-action on $\kk(z)$ to descend to an action of $u_q(\mf{sl}_2)$ is
\[
\begin{array}{rl}
q^2w_1(1-k)w_2(1-k) - w_2(1-k)w_1(1-k)
&= (k^2 -1)(q-q^{-1})^{-1},
\end{array}
\]
i.e., after simplifications on the left hand side, 
\[
w_1w_2(q^2-1)(1-k^2)
= (k^2 -1)(q-q^{-1})^{-1}.
\]
Thus, the condition is $w_1w_2 = -q(q^2-1)^{-2}$. Now, set
$$w_1 = (1-q^{-2})^{-1}z^{-1} \quad \text{and} \quad w_2 = -q^{-1}(q^2-1)^{-1} z$$
to get the action of $e$ and $f$. Namely, we chose $w_1$ as above to get that $e$, which acts as $w_1(1-k)$, satisfies $e \cdot z =1$. On the other hand, $kf$, which acts as $w_2(1-k)$, satisfies $$kf \cdot z ~=~ - q^{-1}(q^2-1)^{-1}(1-q^{-2})z^2~=~-q^{-3} z^2.$$ Hence, $f \cdot z = q^4 kf \cdot z = -qz^2$. This recovers the action from \cite[Proposition~25(2)]{PartI}.
\end{example}


\subsection{The Nichols algebra relations} \label{nar}

We now return to our main problem: determining when a finite-dimensional pointed Hopf algebra $H$ 
with an abelian group of grouplike elements $G$ is Galois-theoretical. 
By Proposition~\ref{prop:Angiono}, $H$ is generated by $G$ and $(g_i,1)$-skew-primitive elements $x_i$, for $i = 1, \dots, \theta$.

Assume that $H$ is coradically graded, i.e., $H$ is the bosonization $\mf{B}(V) \# \kk G$. 
In this case, $H$ is a quotient of $H(G,\ug,\uchi)$ of Definition~\ref{def:H(G,theta)}. The kernel of the surjective homomorphism $\phi: H(G,\ug,\uchi)\to H$ 
is generated by some noncommutative polynomials $P_\alpha$ in the $x_i$, which are relations of 
the Nichols algebra $\mf{B}(V)$. Note that since $H$ is coradically graded, we can choose $P_\alpha$ to be homogeneous of some degree $d_i(\alpha)$ in each $x_i$.

Consider any action of $H(G,\ug,\uchi)$ on $L$ such that $G$ acts faithfully.  
By Theorem~\ref{thm:H(G,theta)}, this action is given by the formula 
$x_i\mapsto w_i(1-g_i)$. Then, $P_\alpha$ acts 
on $L$ by the operator $\prod_i w_i^{d_i(\alpha)}Q_\alpha$, 
where $Q_\alpha$ is an element of $\kk G$. Note that $Q_\alpha$ is independent of the 
choice of the $w_i$ and the choice of module field $L$.

\begin{proposition}\label{criter} Let $H$ be a finite-dimensional, pointed, coradically graded Hopf algebra with an abelian group of grouplike elements $G$. Then,  $H$ is Galois-theoretical if and only if 
the elements $Q_\alpha\in \kk G$ vanish for all $\alpha$.
\end{proposition}  

\begin{proof} 
To prove the forward direction, assume that $H$ is Galois-theoretical, and acts inner faithfully on a field $L$. 
Then, we can pull back the action of $H$ on $L$ to an action of  
$H(G,\ug,\uchi)$, where $P_\alpha$ acts by zero. Since $H$ acts on $L$ inner faithfully, each $x_i$ acts by nonzero. So, each $w_i$ is nonzero. Thus, $Q_\alpha=0$ for all 
$\alpha$, as desired. 

Conversely, fix a faithful action of $G$ on a field $L$. 
Note that since $Q_\alpha=0$, the formula 
$x_i\mapsto w_i(1-g_i)$ defines an $H$-action on $L$. Pick $\{w_i~|~ i \in I_{g,\chi}\}$ 
to be linearly independent; we have shown in Theorem \ref{thm:H(G,theta)} that this can be achieved. 
Then, the $x_i$ act by linearly independent operators on $L$. By  \cite[Corollary~5.4.7]{Montgomery}, any nonzero Hopf ideal in $H$ intersects the $\kk$-span of $\{x_i~|~i\in I_{g,\chi}\}$ 
in $H$ nontrivially. Therefore, the $H$-action on $L$
is inner faithful.   
\end{proof} 

This gives an effective criterion of Galois-theoreticity for coradically graded $H$: the relations $P_\alpha$ are known,
which allows one to  calculate explicitly the elements $Q_\alpha$. 

\begin{example} In contrast with Example~\ref{sl2ex}, take $H={\rm gr}(u_q(\mf{sl}_2))$. We then have a single relation $P_1=q^2x_1x_2-x_2x_1$ of $\mc{B}(V)$ to consider. The corresponding element of $\kk G$ is
$$Q_1= q^2(1-\chi_2(k)k)(1-k) - (1-\chi_1(k)k)(1-k) = (q^2-1)(1-k^2)\ne 0.$$
This shows that 
${\rm gr}(u_q(\mf{sl}_2))$ is not Galois-theoretical due to Proposition~\ref{criter}. 
\end{example}

We will use the Galois-theoreticity criterion above in a number of other examples below.  

\begin{proposition}\label{bothcannot} 
Let $H$ be a finite-dimensional pointed Hopf algebra with an abelian group of grouplike elements $G$. If $H$ is Galois-theoretical, then a nontrivial lifting of $H$ cannot be Galois-theoretical.  
\end{proposition} 

\begin{proof} 
Let $\widetilde{H}$ be a lifting of $H$. 
Assume that $L$ is an inner faithful $\widetilde{H}$-module field. 
Then by Proposition~\ref{anypoi}, $\widetilde{H}$ is a Hopf quotient of $H(G,\ug,\uchi)$. Further,
the action of $\widetilde{H}$ on $L$ descends from an action of 
$H(G,\ug,\uchi)$ on $L$. But since $H$ is Galois-theoretical, we have $Q_\alpha=0$ for all $\alpha$ by Proposition~\ref{criter}. So, 
for any action of $H(G,\ug,\uchi)$ on $L$ the elements $P_\alpha$ act by zero. In particular, the above action of $H(G,\ug,\uchi)$ factors through $H$. So, $\widetilde{H}$ is a Hopf quotient of $H$, of the same dimension as $H$. Hence, $\widetilde{H}=H$ (i.e., the action of $\widetilde H$ on $L$ is, in fact, and action of $H$), that is, $\widetilde{H}$ is a trivial lifting of $H$. 
\end{proof} 

\subsection{Minimal Hopf algebras} \label{subsec:minimal}

Let $H$ be a pointed Hopf algebra with $G=G(H)$ abelian 
(so, generated by grouplike and skew-primitive elements by Proposition~\ref{prop:Angiono}), and consider the following notation.
\medskip

\noindent {\it Notation} [$G', H'$].
Let $G'\subset G$ be the subgroup generated by all $g\in G$ for which there is a nontrivial $(g,1)$-skew-primitive 
element in $H$.  Let $H'$ be the subalgebra of $H$ generated by $G'$ and 
all of the $(g,1)$-skew-primitive elements for $g\in G'$. 
\medskip

Clearly, $G'$ is a normal subgroup of $G$ and $H=\kk G\otimes_{\kk G'}H'$. 

\begin{definition} \label{def:minimal} 
We say that $H$ is {\it minimal} if $H=H'$.
\end{definition} 

\begin{example} \label{ex:minimal} We have the following examples of minimal pointed Hopf algebras from \cite{PartI}:
\begin{enumerate}
\item Taft algebras $T(n,\zeta)$ with $G' = \langle g \rangle \cong \mathbb{Z}_n$; 
\item $E(n)$  with $G' = \langle g \rangle \cong \mathbb{Z}_2$; 
\item generalized Taft algebras $T(n,m,\alpha)$  with $G' = \langle g \rangle \cong \mathbb{Z}_n$; 
\item book algebras $\mathbf{h}(\zeta,p)$  with $G'=\langle g \rangle = \mathbb{Z}_n$;
\item $H_{3^4}$  with $G' = \langle g \rangle = \mathbb{Z}_3$; and
\item $u_q(\mf{sl}_2)$  with $G'=\langle k \rangle = \mathbb{Z}_{m}$.
\end{enumerate}
\end{example}

To study the Galois-theoretical properties of pointed coradically graded Hopf algebras with $G(H)$ abelian (as in Sections~\ref{sec:coradA1r} and~\ref{sec:rank2}), we may focus on the case of $H$ being minimal due to the following result. 

\begin{proposition} \label{prop:minimal} Suppose $H$ is a finite-dimensional, pointed, coradically graded Hopf algebra with $G=G(H)$ abelian. 
Then, $H$ is Galois-theoretical if and only if so is the minimal Hopf subalgebra $H'$. 
\end{proposition}

\begin{proof} 
Using the notation from the beginning of Subsection~\ref{nar}, we get that the elements $Q_\alpha$ for both $H$ and $H'$ are the same. So by applying Proposition~\ref{criter} to  both $H$ and $H'$, we get the desired result.
\end{proof}


\section{Coradically graded Galois-theoretical Hopf algebras, type A$_1^{\times \theta}$} \label{sec:coradA1r}

We begin this section by classifying finite-dimensional, pointed, Galois-theoretical Hopf algebras of rank one (of type A$_1$).
We then study the Galois-theoretical property of pointed coradically graded Hopf algebras of finite Cartan type A$_1^{\times \theta}$; these are known as {\it bosonizations of quantum linear spaces}.

\subsection{Type A$_1$: Galois-theoretical Hopf algebras of rank one} \label{subsec:rank1}

We determine precisely when a finite-dimensional pointed Hopf algebra $H$ of rank one is Galois-theoretical. We also determine in this case the structure of the $H$-module fields $L$.
The classification of finite-dimensional pointed Hopf algebras of rank one is provided in \cite{KropRadford}; we repeat their result below.

\begin{theorem} \label{thm:KR} \cite[Theorem~1(a)]{KropRadford} Let $G$ be a finite group with character map $\chi:G \rightarrow \kk^{\times}$, and take $g \in Z(G)$ and $\alpha \in \kk$. Every finite-dimensional pointed Hopf algebra of rank one is generated by $G$ and a $(g,1)$-skew-primitive element $x$, subject to the group relations of $G$ and the relations
$$x^m = \alpha(g^m -1)  \quad \text{ and } \quad ax = \chi(a) xa$$
for all $a \in G$. \qed
\end{theorem}

Note that in this situation, $m$ is the order of the root of unity $\chi(g)$. 
The main result for the rank one case is as follows.

\begin{theorem} \label{thm:rank1} Let $H$ be a finite-dimensional pointed Hopf algebra of rank one with $G=G(H)$, not necessarily abelian.
Then: 
\begin{enumerate}
\item $H$ is Galois-theoretical if and only if the Hopf subalgebra $H'$ generated by $\{g,x\}$ is a Taft algebra $T(n,\zeta)$. In this case, if $H$ acts inner faithfully on a field 
$L$, then 
\begin{enumerate}
\item[(i)] $L$ is a Galois extension of the field $F = L^G=L^H$; 
\item[(ii)] $L$ is also a cyclic extension of degree $n$ of the intermediate field $E=L^{T(n,\zeta)}= L^{\mathbb{Z}_n}$.
\end{enumerate}
\item Moreover, $H$ acts inner faithfully on any field containing $\kk$ which admits a faithful action of $G$, so that $x$ acts by nonzero.
\end{enumerate}
\end{theorem}

\begin{proof}
(a) First, let us assume that $H$ is Galois-theoretical. Consider the Hopf subalgebra $H'$ of $H$ generated by $g$ and $x$.
We know by Theorem~\ref{thm:KR} that $H'$ is a generalized Taft algebra $T(n, m,\alpha)$, where $n \in \mathbb{Z}_+$ so that $m$ divides $n$. Now, by  \cite[Propositions~10(3) and~21]{PartI}  we must have that
$H'$ is an ordinary Taft algebra $T(n,\zeta)$. 

Conversely, assume that the minimal Hopf subalgebra $H'$ generated by $g$ and $x$ is a Taft algebra. Then, $H'$ is Galois theoretical by \cite[Proposition~17]{PartI}. Thus, $H$ is Galois-theoretical by Proposition~\ref{prop:minimal}. 

(i,ii) Apply classical Galois theory and  \cite[Theorem~11]{PartI}.

(b)  Apply Corollary~\ref{cor:GnotAbel}.
\end{proof}


\subsection{Type A$_1^{\times \theta}$: bosonizations of quantum linear spaces} \label{subsec:B(G,theta)}

According to \cite[Theorem~5.5]{AndSch:p3}, every finite-dimensional pointed coradically graded Hopf algebra of finite Cartan type A$_1^{\times \theta}$ is isomorphic to a bosonization of a {\it quantum linear space}. The latter is a braided Hopf algebra in ${}^G_G \mc{YD}$ defined in \cite[Section~3]{AndSch:p3}; its bosonization by $\kk G$ is defined as follows.

\begin{definition} \cite{AndSch:p3} \label{def:B(G,theta)} Let $\theta \geq 1$ and let $G$ be a finite abelian group.
The {\it bosonization of a quantum linear space} is a  Hopf algebra $$B(G, \ug, \uchi) := B(G,g_1,\dots,g_\theta,\chi_1,\dots,\chi_\theta),$$  generated by $G$ and $(g_i, 1)$-skew-primitive $x_i$ for $g_i \in G$, with $i =1, \dots, \theta$. Given characters $\chi_1, \dots, \chi_{\theta} \in \widehat{G}$ with $q_{ii}:=\chi_i(g_i)$ having orders $n_i \geq 2$, we have that $B(G, \ug, \uchi) $ is subject to the relations of $G$ and 
$$g x_i = \chi_i(g) x_ig,  \quad \quad x_i^{n_i} = 0, \quad \quad x_i x_j = \chi_j(g_i) x_j x_i,$$
for all $g \in G$ and $i \neq j$. We also assume that
\begin{equation} \label{eq:B(G,theta)}
\chi_j(g_i) \chi_i(g_j) = 1 
\end{equation} 
for all $i \neq j$. 
\end{definition}

Note that $B(G, \ug, \uchi) $ is a finite-dimensional Hopf algebra quotient of the Hopf algebra $H(G, \ug, \uchi)$  defined in Section~\ref{subsec:H(G,theta)}. Moreover, \eqref{eq:B(G,theta)} implies that $B(G, \ug, \uchi)$ is of Cartan type A$_1^{\times \theta}$.

\begin{theorem} \label{thm:B(G,theta)} 
The Hopf algebra $B(G, \ug, \uchi) $ is Galois-theoretical if and only if the minimal Hopf subalgebra $B'$ of $B(G, \ug, \uchi) $ generated  by $\{g_1, \dots g_{\theta}, x_1, \dots, x_{\theta}\}$ is the tensor product of Taft algebras $T(n,\zeta)$,  Nichols Hopf algebras $E(n)$, and book algebras $\mathbf{h}(\zeta,1)$.
In this case, if $L$ is an inner faithful $B(G, \ug, \uchi)$-module field, then:
\begin{enumerate}
\item[(i)] $L$ is a Galois extension of the intermediate field $F = L^B=L^G$; and
\item[(ii)] $L$ is also a Galois extension of the intermediate field $E=L^{B'}=L^{G'}$, where $G'$ is the subgroup of $G$ generated by $g_1, \dots, g_{\theta}$.
\end{enumerate}
\end{theorem}

\begin{proof} 
If $B'$ is the tensor product of $T(n,\zeta)$, $E(n)$, or $\mathbf{h}(\zeta,1)$, then $B'$ is Galois-theoretical by \cite[Propositions~~10(5),~17,~19,~22]{PartI}. Now, $B(G, \ug, \uchi) $ is Galois-theoretical by Proposition~\ref{prop:minimal}.

Conversely, if $B(G, \ug, \uchi) $ is Galois-theoretical with inner faithful module field $L$, then by Theorem~\ref{thm:H(G,theta)} and Remark~\ref{rem:w_i}, $x_i$ acts as ${w_i}(1-g_i)$.  Thus for $i \neq j$,  the relations
$g_i {w_j} = \chi_j(g_i) {w_j} g_i$, $x_i x_j - \chi_j(g_i) x_j x_i = 0$, and \eqref{eq:B(G,theta)} imply that
\[
\begin{array}{l}
{w_i}(1-g_i){w_j}(1-g_j) - \chi_j(g_i){w_j}(1-g_j){w_i}(1-g_i)\\
~=~ {w_i}{w_j}\left[(1-\chi_j(g_i)g_i)(1-g_j) - \chi_j(g_i)(1-\chi_i(g_j)g_j)(1-g_i)\right] \\
~=~{w_i}{w_j}[(1-\chi_j(g_i))(1-g_ig_j)]~=~ 0.
\end{array}
\]
Using the notation of Subsection \ref{nar}, we have $Q_{ij}=(1-\chi_j(g_i))(1-g_ig_j)$. 
So, $Q_{ij}=0$ by Proposition~\ref{criter}, and  either $\chi_j(g_i) = 1$ or $g_ig_j=1$ for all $i \neq j$.

To proceed, define a graph $\Gamma$ whose vertices are labelled $1, \dots ,\theta$, 
where we have an edge $i$---$j$ if and only if $\chi_j(g_i)\neq 1$. In this case, $g_ig_j=1$. 

Then, any connected component $C$ of $\Gamma$ has 
{\it order} equal to the order of an element $g_i$ for $i\in C$; this does not depend on $i$. (This value is also the order of $\chi_i(g_i)$ and nilpotency order of $x_i$, as the Hopf subalgebra generated by $\{g_i, x_i\}$ is a Taft algebra by \cite[Proposition~10(3)]{PartI} and Theorem~\ref{thm:rank1}(a).) We let $H(C)$ be the minimal Hopf subalgebra of $B(G, \ug, \uchi) $ generated by $\{ g_i, x_i \}_{i \in C}$.
 
\begin{sublemma} \label{sublem:QLS1} Suppose $\Gamma$ is connected and has order 2. Then, $\Gamma$ is a complete graph, and 
$H(\Gamma)=E(\theta)$.
\end{sublemma}

\noindent {\it Proof of Sublemma~\ref{sublem:QLS1}}. 
Suppose that there are edges $i$---$j$ and $j$---$r$ in $\Gamma$, so we have $\chi_j(g_i) \ne 1$ and $\chi_r(g_j)\ne 1$. Then $g_i=g_j$, as $\Gamma$ has order 2. Moreover, 
$\chi_r(g_i)=\chi_r(g_j)\ne 1$, so $i$---$r$. Thus, $\Gamma$ is a complete graph. 
Since $\Gamma$ has order 2, we have that ord($g_i$) $=2$.  
Now since $g_i g_j =1$, we get that  $g_i=g_j$ for all $i \neq j$. Thus, $H=E(\theta)$. 

\begin{sublemma} \label{sublem:QLS2}
Suppose  that $\Gamma$ is connected, has more than one vertex, and has order $n>2$. 
Then, $\Gamma$ has two vertices, and $H(\Gamma)$ is a book algebra $\mathbf{h}(\zeta,1)$.\end{sublemma}

\noindent {\it Proof of Sublemma~\ref{sublem:QLS2}}. 
Suppose that we have $i$---$j$---$r$ in $\Gamma$. Then, let $g:=g_i$, so $g_j=g^{-1}$ and $g_r=g$. Hence,
\begin{equation} \label{eq:chi si}
\chi_s(g_i)=\chi_s(g_j)^{-1}=\chi_s(g_r)
\end{equation}
 for all $s$. 
Thus, by \eqref{eq:B(G,theta)} and \eqref{eq:chi si},
$$\chi_j(g_i)=\chi_j(g_j)^{-1}=\chi_j(g_r)=\chi_r(g_j)^{-1}=\chi_r(g_i)=\chi_i(g_r)^{-1}=\chi_i(g_j)
=\chi_j(g_i)^{-1},$$ so $\chi_j(g_i)$ has order 2. This yields a contradiction, as $\chi_j(g_i)=\chi_j(g_j)^{-1}$ by \eqref{eq:chi si}, 
which has order $n >2$. Thus, $\Gamma$ has two vertices, say 1 and 2, connected with an edge. Set $g_1=g$, $g_2=g^{-1}$ and $x_1=y$,  $x_2=g^{-1}x$, so that
$\Delta(x)=1\otimes x+x\otimes g$, $\Delta(y)=g\otimes y+y\otimes 1$. Moreover,  take $\chi_1(g_2) = \zeta^{-1}$ for $\zeta$ a primitive $n$-th root of unity, so that $\chi_2(g_1) = \zeta$ and $xy=yx$. Now, $H(\Gamma)$ is the book algebra $\mathbf{h}(\zeta,1)$ generated by $g,x,y$. 

\begin{sublemma} \label{sublem:QLS3} We have that $B' \cong H(C_1)\otimes\dots\otimes H(C_m)$, where $C_j$ are the components of $\Gamma$.
\end{sublemma}

\noindent{\it Proof of  Sublemma~\ref{sublem:QLS3}}. We only need to show that $G(B')=G(H(C_1))\times \cdots \times G(H(C_m))$. Suppose that we have a relation $h_1 \cdots h_m=1$, where 
$h_j\in G(H(C_j))$. Our job is to show that $h_j=1$ for all $j$. 
To do so, note that for $h_r = \prod_{t \in C_r} g_t^{p_t}$, we get $\chi_i(g_t)=1$ for $i \in C_j$, $t \in C_r$ and $j \neq r$. Hence, $\chi_i(h_r) =1$ for $i \in C_j$ with $j \neq r$. Since $h_1 \cdots h_m=1$, we get that $h_j = \prod_{r \neq j} h_r^{-1}$ and $\chi_i(h_j) =1$ for all $j$. This implies that $h_j=1$, as desired.
\medskip

By Sublemmas~\ref{sublem:QLS1} and~\ref{sublem:QLS2}, $H(C_j)$ is $E(n)$, a book algebra, or a Taft algebra; the latter occurs if $C_j$ has only one vertex. Hence, if $B(G, \ug, \uchi)$ is Galois-theoretical, then $B'$ is as claimed by Sublemma~\ref{sublem:QLS3}.

\medskip

(i,ii) Apply classical Galois theory and  \cite[Theorem~11]{PartI}.
\end{proof}


\section{Coradically graded Galois-theoretical Hopf algebras, rank two} \label{sec:rank2}

In this section, we study the Galois-theoretical properties of finite-dimensional pointed Hopf algebras of rank two that are coradically graded. Recall we assume that $G=G(H)$ is an abelian group, and as a result,  $H \cong \mf{B}(V) \# \Bbbk G$  with braiding matrix $(q_{ij}= \chi_j(g_i))$.

The finite Cartan types of rank two are A$_1 \times$A$_1$, A$_2$, B$_2$, or G$_2$, with corresponding Cartan matrices $(a_{ij})$:
\[
\begin{array}{ccccccc}
\begin{pmatrix} 2 & 0\\0 &2\end{pmatrix} &&
\begin{pmatrix} 2 & -1\\-1 &2\end{pmatrix} &&
\begin{pmatrix} 2 & -2\\-1 &2\end{pmatrix} &&
\begin{pmatrix} 2 & -1\\-3 &2\end{pmatrix}\\
\text{A}_1 \times \text{A}_1 &&
\text{A}_2 && \text{B}_2 && \text{G}_2
\end{array}
\]

\medskip

We have the following lemma.

\begin{lemma}\label{lem:rank2}
Let $H\cong\mf{B}(V) \#\kk G$ be a finite-dimensional, pointed, coradically graded Hopf algebra of rank two, not of type A$_1 \times$A$_1$, subject to \eqref{eq:conditions}. Let $\{x_i\}_{i=1,2}$ be a basis of  $(g_i,1)$-skew-primitive elements for the graded braided vector space $V$. Let $H'$ be the minimal Hopf subalgebra generated by $g_1, g_2, x_1, x_2$. 

Then, $H'$ is a finite-dimensional Hopf algebra quotient of $H( \langle g_1, g_2 \rangle, g_1, g_2, \chi_1, \chi_2)$ from Section~\ref{subsec:H(G,theta)} subject to relations:
$$ \text{ad}_c(x_i)^{1-a_{ij}}(x_j) = 0 \text{ for all } i \neq j \quad \quad \text{ and } \quad \quad
x_\alpha^n=0,
$$
for non-simple roots $\alpha$,
where $x_\alpha$ are the Cartan-Weyl root elements, and $n = \text{ord}(g_1) = \text{ord}(g_2)$.  
 Here, $(\text{ad}_c x_i)(y) = x_i y - (q_{ij_1} \cdots q_{ij_t}) y x_i$
for $y = x_{j_1} \cdots x_{j_t}$. 
\end{lemma}

\begin{proof}
This is a special case of Theorem~\ref{thm:finiteCartan}. In particular, $n = \text{ord}(g_1) = \text{ord}(g_2)$ as the Hopf subalgebras generated by $\{g_i, x_i\}$, for $i =1$ or $2$, are Taft algebras.
\end{proof}

The exact form of $x_\alpha$ will not be important for us, due to the following result.

\begin{lemma} \label{lem:Qis0} Retain the notation from Theorem~\ref{thm:finiteCartan} and Subsection~\ref{nar}.  Suppose that each $g_i$ is a power of a grouplike element $g$ of order $n$, and $gx_\alpha=\zeta x_\alpha g$ for a primitive $n$-th root of unity $\zeta$. Then the element $Q_{\alpha} \in \kk G$ corresponding to the relation $x_{\alpha}^n$ is equal to 0. 
\end{lemma}

\begin{proof}
Let $L$ be a module field for $H(G,\underline{g},\underline{\chi})$. By Theorem~\ref{thm:H(G,theta)}, $x_{\alpha}$ acts on $L$ by $wQ(g)$, where $w$ is a monomial in terms of the $w_i$, and $Q(g) \in \kk[g]/(g^n-1)$. Hence, the expression $x_{\alpha}^n$ acts on $L$ by $(w Q(g))^n$, which is $w^n \prod_{i=0}^{n-1} Q(\zeta^i g)$. Therefore, $Q_{\alpha} = \prod_{i=0}^{n-1} Q(\zeta^i g)$.

By Theorem~\ref{thm:H(G,theta)}, $Q(g)$ is a multiple of $1-g$. Hence, we have that $Q_\alpha$ is a multiple of $\prod_{i=0}^{n-1}(1-\zeta^i g)$. But $\prod_{i=0}^{n-1}(1-\zeta^i g) = 1-g^n=0$, so we get the desired result.
\end{proof}

The proof of the main result, Theorem~\ref{thm:ranktwo}, boils down to two cases: (a) $q_{12} =1$ or $q_{21}=1$, and (b) $g_1^2g_2=1$ or $g_1g_2^2=1$. We proceed in case (b) below.

\begin{proposition} \label{prop:g1^2g2=1}
Retain the notation from Lemma~\ref{lem:rank2}. Assume that $q_{12}, q_{21} \neq 1$.
If $g_1^2g_2=1$ or $g_1g_2^2=1$, then $H'$ is Galois-theoretical if and only if either
\begin{itemize}
\item $H'$ is one of the $3^4$-dimensional Hopf algebras $H_{3^4}$ of type A$_2$ from
\cite[Theorem~1.3(ii)]{AndSch:quantum};
\item $H'$ is one of the $5^5$-dimensional Hopf algebras $H_{5^5}$ of type B$_2$ from
\cite[Theorem~1.3(iii)]{AndSch:quantum}; or
\item $H'$ is one of the $7^7$-dimensional Hopf algebras $H_{7^7}$ of type G$_2$ from
\cite[Theorem~1.3(iv)]{AndSch:quantum}.
\end{itemize}
\end{proposition}

\begin{proof} By \eqref{eq:qij Cartan}, we have that $q_{11}^{-a_{12}}q_{12}q_{21}=1$ and $q_{12}q_{21}q_{22}^{-a_{21}} =1$. So there exists a unique primitive $n$-th root of unity $\zeta$, and a scalar $\lambda \in \kk$, such that 
$q_{11} = \zeta^{a_{21}}, ~ q_{12} = \zeta^{a_{12}a_{21}}\lambda^{-1}, ~q_{21}= \lambda, ~ q_{22} = \zeta^{a_{12}}.$

If $g_1^2g_2=1$, then we can take $g := g_1$ so that $g_2=g^{-2}$, where ord($g_1$)= $n\geq3$. 
Now 
\[g x_1 = \zeta^{a_{21}} x_1 g, \hspace{.3in} g^{-2}x_1 = \lambda x_1g^{-2} , \hspace{.3in} g x_2 = \zeta^{a_{12}a_{21}}\lambda^{-1} x_2 g,  \hspace{.3in}  g^{-2}x_2 = \zeta^{a_{12}} x_2g^{-2}.\]
The first two equations yield $\lambda = (\zeta^{a_{21}})^{-2}$, and the last two equations yield $\zeta^{a_{12}} = (\zeta^{a_{12}a_{21}} \lambda^{-1})^{-2}$.
Thus, 
\begin{equation} \label{eq:zetaaij}
\zeta^{a_{12}} = \zeta^{-2a_{12}a_{21}-4a_{21}}.
\end{equation}

\noindent In type A$_2$, we have that $a_{12} = a_{21} = -1$. So, $n =3$ by \eqref{eq:zetaaij} and we get  by \cite[Theorem~1.3(ii)]{AndSch:quantum} that $H'$ is an $3^4$-dimensional Hopf algebra $H_{3^4}$ of type A$_2$. In type B$_2$, we have that $a_{12} = -2$ and  $a_{21} = -1$. So, $n =2$  by \eqref{eq:zetaaij}, which contradicts \eqref{eq:conditions}.
In type G$_2$, we have that $a_{12} = -1$ and $a_{21} = -3$. So, $n=7$ 
 and $H'$ is the bosonization of one of the Nichols algebras from \cite[Theorem~1.3(iv)]{AndSch:quantum}.

If $g_1g_2^2=1$,  then we can take $g := g_2$ so that $g_1=g^{-2}$. Similarly, one gets 
\begin{equation} \label{eq:zetaaij2}
\zeta^{a_{21}} = \zeta^{-2a_{12}a_{21}-4a_{12}}.
\end{equation}
In type A$_2$, we have that $n =3$ by \eqref{eq:zetaaij2}. So again, $H'$ is an $3^4$-dimensional Hopf algebra $H_{3^4}$ of type A$_2$.
In type B$_2$, we have that $n =5$ by \eqref{eq:zetaaij2}. Therefore, $H'$ is a bosonization of one of the Nichols algebras from \cite[Theorem~1.3(iii)]{AndSch:quantum}.
In type G$_2$, we have that $n =1$ by \eqref{eq:zetaaij2}, which yields a contradiction.

Hence, if $g_1^2g_2=1$ or $g_1g_2^2=1$, then $H'$ is either one of the
\begin{enumerate}
\item[(i)]  $3^4$-dimensional Hopf algebras $H_{3^4}$ of type A$_2$ from
\cite[Theorem~1.3(ii)]{AndSch:quantum} where $g_1^2g_2=g_1g_2^2=1$; 
\item[(ii)]  $5^5$-dimensional Hopf algebras $H_{5^5}$ of type B$_2$ from \cite[Theorem~1.3(iii)]{AndSch:quantum} where $g_1g_2^2=1$; or
\item[(iii)] $7^7$-dimensional Hopf algebras $H_{7^7}$ of type G$_2$ from \cite[Theorem~1.3(iv)]{AndSch:quantum} where $g_1^2g_2=1$.
\end{enumerate}
Now let us show that each of the Hopf algebras above is Galois-theoretical.  

First, the Hopf algebra in (i) has braiding matrix $q_{11} = q_{12} = q_{21} = q_{22} =$ a primitive third root of unity, and is Galois-theoretical by \cite[Proposition~23]{PartI}. 

We use Proposition~\ref{criter} to check that the Hopf algebra in (ii) is Galois-theoretical. Namely, we need each of the elements $Q_{\alpha} \in \kk G$ corresponding to relations of the Hopf algebra to be $0$.
We take $g_1 = g^3$ and $g_2 = g$, where ord($g$) = ord($\zeta$) = 5. Apply Lemma~\ref{lem:Qis0} to conclude that the elements $Q_{\alpha}$ corresponding to the relations $x_{\alpha}^5 =0$ are 0. The remaining relations of the Hopf algebra in this case are
\[
\begin{array}{l}
ad_c(x_1)^3(x_2)  = x_1^3x_2 - (q_{11}^2q_{12}+q_{11}q_{12}+q_{12})x_1^2x_2x_1 
+(q_{11}^3q_{12}^2+q_{11}^2q_{12}^2+q_{11}q_{12}^2) x_1x_2x_1^2  
-q_{11}^3q_{12}^3x_2x_1^3 = 0, \\
ad_c(x_2)^2(x_1) = x_2^2x_1 -(q_{21}q_{22}+q_{21})x_2x_1x_2  +(q_{21}^2q_{22})x_1x_2^2 = 0.
\end{array}
\]

To compute the element $Q_{12} \in \kk G$ corresponding to the relation $\text{ad}_c(x_1)^3(x_2)$, we use Theorem~\ref{thm:H(G,theta)} and Remark~\ref{rem:w_i}. For instance, $x_1^3x_2$ acts as 
$$[w_1(1-g_1)]^3 w_2(1-g_2) = w_1^3w_2[(1-q_{11}^2q_{12}g_1)(1-q_{11}q_{12}g_1)(1-q_{12}g_1)(1-g_2)],$$ so $(1-q_{11}^2q_{12}g_1)(1-q_{11}q_{12}g_1)(1-q_{12}g_1)(1-g_2)$ is a summand of $Q_{12}$.
Now we obtain that the element $Q_{12}$ is 0 by the Maple code below:

{\footnotesize
\begin{verbatim}
Q12:= (1-q11^2*q12*g1)*(1-q11*q12*g1)*(1-q12*g1)*(1-g2)
      -(q11^2*q12+q11*q12+q12)*(1-q11^2*q12*g1)*(1-q11*q12*g1)*(1-q21*g2)*(1-g1)
      +(q11^3*q12^2+q11^2*q12^2+q11*q12^2)*(1-q11^2*q12*g1)*(1-q21^2*g2)*(1-q11*g1)*(1-g1)
      -q11^3*q12^3*(1-q21^3*g2)*(1-q11^2*g1)*(1-q11*g1)*(1-g1);
g1:=g^3;              g2:=g;                      zeta:=exp((2/5)*Pi*I);
q11:=zeta^(-1);       q12:=zeta^2*lambda^(-1);    q21:=lambda;            q22:=zeta^(-2);
collect(simplify(Q12),[g],'distributed');
## The only powers of g that arise are 0, 10.
simplify(coeff(Q12,g,0)+coeff(Q12,g^(10)));
>>                                      0
\end{verbatim}
}

\noindent Similarly, we obtain that the element $Q_{21}$, corresponding to the relation $ad_c(x_2)^2(x_1)$, is 0 as follows:

{\footnotesize
\begin{verbatim}
Q21:= (1-q21*q22*g2)*(1-q21*g2)*(1-g1)
      -(q21*q22+q21)*(1-q21*q22*g2)*(1-q12*g1)*(1-g2)
      +(q21^2*q22)*(1-q12*q12*g1)*(1-q22*g2)*(1-g2);
g1:=g^3;             g2:=g;                      zeta:=exp((2/5)*Pi*I);
q11:=zeta^(-1);      q12:=zeta^2*lambda^(-1);    q21:=lambda;            q22:=zeta^(-2);
collect(simplify(Q21),[g],'distributed');
## The only powers of g that arise are 0, 5.
simplify(coeff(Q21,g,0)+coeff(Q21,g^(5)));
>>                                      0
\end{verbatim}
}

Now we use Proposition~\ref{criter} to check that the Hopf algebra in (iii) is Galois-theoretical, in the same fashion as above. 
We take $g_1 = g$ and $g_2 = g^5$, where ord($g$) = ord($\zeta$) = 7. 
Apply Lemma~\ref{lem:Qis0} to conclude that the elements $Q_{\alpha}$ corresponding to the relations $x_{\alpha}^7 =0$ are 0. 
We obtain that the elements $Q_{12}$ and $Q_{21}$, corresponding to remaining relations $ad_c(x_1)^2(x_2)$ and $ad_c(x_2)^4(x_1)$, respectively, are 0 by the following Maple code:

{\footnotesize
\begin{verbatim}
Q12:= (1-q11*q12*g1)*(1-q12*g1)*(1-g2)
      -(q11*q12+q12)*(1-q11*q12*g1)*(1-q21*g2)*(1-g1)
      +(q11*q12^2)*(1-q21*q21*g2)*(1-q11*g1)*(1-g1);
 
Q21:= (1-q21*q22^3*g2)*(1-q21*q22^2*g2)*(1-q21*q22*g2)*(1-q21*g2)*(1-g1)
      -(q21*q22^3+q21*q22^2+q21*q22+q21)* (1-q21*q22^3*g2)*(1-q21*q22^2*g2)*(1-q21*q22*g2)*(1-q12*g1)*(1-g2)                                                          
      +(q21^2*q22^5+q21^2*q22^4+q21^2*q22^3+q21^2*q22^3+q21^2*q22^2+q21^2*q22)
           *(1-q21*q22^3*g2)*(1-q21*q22^2*g2)*(1-q12^2*g1)*(1-q22*g2)*(1-g2)
      -(q21^3*q22^6+q21^3*q22^5+q21^3*q22^4+q21^3*q22^3)
           *(1-q21*q22^3*g2)*(1-q12^3*g1)*(1-q22^2*g2)*(1-q22*g2)*(1-g2)
      +(q21^4*q22^6)*(1-q12^4*g1)*(1-q22^3*g2)*(1-q22^2*g2)*(1-q22*g2)*(1-g2);

g1:=g;                g2:=g^5;                    zeta:=exp((2/7)*Pi*I);
q11:=zeta^(-3);       q12:=zeta^3*lambda^(-1);    q21:=lambda;            q22:=zeta^(-1);

collect(simplify(Q12),[g],'distributed');
## The only powers of g that arise are 0, 7.
simplify(coeff(Q12,g,0)+coeff(Q12,g^(7)));
>>                                      0

collect(simplify(Q21),[g],'distributed');
## The only powers of g that arise are 0, 21.
simplify(coeff(Q21,g,0)+coeff(Q21,g^(21)));
>>                                      0
\end{verbatim}
}
Thus, by Proposition~\ref{criter}, the Hopf algebras in (ii) and (iii) are Galois-theoretical as well.
\end{proof}

Now let us deal with case (a) discussed before Proposition~\ref{prop:g1^2g2=1}, i.e., $q_{12}=1$ or $q_{21}=1$. 
Let $\widetilde{u}_q^{\geq 0}(\mathfrak{g})$ be the positive part of the small quantum group 
of adjoint type, see \cite[Definition~12]{PartI}. Note that in our situation (of rank 2), we have 
$\widetilde{u}_q^{\geq 0}(\mathfrak{g})=u_q^{\geq 0}(\mathfrak{g})$
unless we are in type $A_2$ and $n$ is divisible by $3$. 

\begin{proposition} \label{prop:q12=1}
Retain the hypotheses of Lemma~\ref{lem:rank2}. Suppose that $q_{12}=1$ or $q_{21} =1$. Then, 
\begin{enumerate}
\item $H'$ is isomorphic to a Drinfeld twist $\widetilde{u}_q^{\geq 0}(\mathfrak{g})^{J}$, where $\mathfrak{g}$ is a finite-dimensional simple Lie algebra of type A$_2$, B$_2$, or G$_2$, described in 
\cite[Definition~12]{PartI}; and
\item $H'$ is Galois-theoretical.
\end{enumerate}
\end{proposition}

\begin{proof} (a) It is easy to see that the subgroup generated by $g_1,g_2$ has to be isomorphic to $\Bbb Z_n\times \Bbb Z_n$. The rest is a straightforward verification; namely, the twist $J$ can be chosen so that the skew primitive elements of $\widetilde{u}^{\geq 0}_q(\mathfrak{g})^J$ have corresponding braiding parameters where either $q_{12}$ or $q_{21}$ is equal to 1.

(b) This follows from \cite[Proposition~38]{PartI} and part (a). 
\end{proof}

Now we state and prove the main result of this subsection.

\begin{theorem} \label{thm:ranktwo}
Let $H$ be a finite-dimensional, pointed, coradically graded Hopf algebra of rank two, subject to \eqref{eq:conditions}. Let $\{x_i\}_{i=1,2}$ be a basis of  $(g_i,1)$-skew-primitive elements for the graded braided vector space $V$.
Then, $H$ is Galois-theoretical if and only if the minimal Hopf subalgebra $H'$ of $H$ generated by $g_1, g_2, x_1, x_2$ is either
\begin{enumerate}
\item of type A$_1 \times$A$_1$, namely
\begin{itemize}
\item the tensor product of Taft algebras  $T(n,\zeta) \otimes T(n',\zeta')$ for $n, n' \geq 2$, 
\item the 8-dimensional Nichols Hopf algebra $E(2)$, or 
\item the book algebra $\mathbf{h}(\zeta,1)$;
\end{itemize}
\item of type A$_2$, B$_2$, or G$_2$ with $q_{12} =1$ or $q_{21}=1$ \textnormal{(}here, $H'$ is isomorphic to a twist $\widetilde{u}_q^{\geq 0}(\mathfrak{g})^{J}$\textnormal{)}; or
\item of type A$_2$, B$_2$, or G$_2$ with $q_{12}, q_{21} \neq 1$, where $H'$ is one of the
\begin{itemize}
\item  $3^4$-dimensional Hopf algebras of type A$_2$ from \cite[Theorem~1.3(ii)]{AndSch:quantum} \\ \textnormal{(}here, $g_1^2g_2=g_1g_2^2=1$ and $ord(g_i) =3$\textnormal{)},
\item $5^5$-dimensional Hopf algebras $H_{5^5}$ of type B$_2$ from
\cite[Theorem~1.3(iii)]{AndSch:quantum} \\\textnormal{(}here, $g_1g_2^2=1$ and $ord(g_i) =5$\textnormal{)}, or
\item $7^7$-dimensional Hopf algebras $H_{7^7}$ of type G$_2$ from
\cite[Theorem~1.3(iv)]{AndSch:quantum}\\ \textnormal{(}here, $g_1^2g_2=1$ and $ord(g_i) =7$\textnormal{)}.
\end{itemize}
\end{enumerate}
\end{theorem}

\begin{proof}
Both directions for part (a) hold  by Theorem~\ref{thm:B(G,theta)}. So, we only need to study types A$_2$, B$_2$, G$_2$.
Consider the Serre relation of $H$ given below:
\begin{equation} \label{eq:part(a)Serre}
\text{ad$_c(x_i)^{1-a_{ij}}(x_j) = 0$  \hspace{.1in} if $i \neq j$}.
\end{equation} 
Recall that $(\text{ad}_c x_i)(y) = x_i y - q_{ij_1} \cdots q_{ij_t} y x_i$
for $y = x_{j_1} \cdots x_{j_t}$. 

For types A$_2$ and G$_2$, we have that  \eqref{eq:part(a)Serre} yields the following Serre relation for $H$:
\begin{align*} \label{eq:Serre relation}
x_1^2x_2 - (q_{11}q_{12} + q_{12})x_1 x_2 x_1 + q_{12}^2q_{11}x_2 x_1^2 = 0.
\end{align*}
If $H$ is Galois-theoretical with module field $L$, then apply Theorem~\ref{thm:H(G,theta)}, Remark~\ref{rem:w_i}, and \eqref{eq:qij Cartan} to conclude that 
\begin{equation} \label{eq:uchii uchij}
\begin{array}{l} 
{w_1}(1-g_1){w_1}(1-g_1){w_2}(1-g_2) -(q_{11}q_{12} + q_{12}){w_1}(1-g_1){w_2}(1-g_2){w_1}(1-g_1)\\
\quad +q_{12}^2q_{11}{w_2}(1-g_2){w_1}(1-g_1){w_1}(1-g_1)~=~0.
\end{array}
\end{equation}
Using the notation of Subsection~\ref{nar}, direct computation with $g_i x_j = \chi_j(g_i) x_j g_i$, $\chi_j(g_i) = q_{ij}$,  and \eqref{eq:uchii uchij} yields
\begin{equation} \label{eq:uchi 2}
\begin{array}{rl}
Q_{12}=(q_{11}q_{12}q_{21}-q_{11}q_{12})(q_{12}-q_{11}q_{12}q_{21})g_1^2g_2 &+ q_{12}(q_{21}-1)(q_{11}+1)(q_{11}q_{12}q_{21}-1)g_1g_2\\
 &+ (q_{11}q_{12}q_{21}-1)(1-q_{12}q_{21})g_2\\
 &+ (q_{11}q_{12}-1)(q_{12}-1) \quad =~0.
\end{array}
\end{equation}
That is, we employed Proposition~\ref{criter}.
By \eqref{eq:qij Cartan}, we have that $q_{11}q_{12}q_{21}=1$. So, \eqref{eq:uchi 2} implies that
$$(1-q_{11}q_{12})(q_{12}-1)(g_1^2g_2 - 1) = 0.$$
So again by \eqref{eq:qij Cartan}, $q_{12} = 1$ or $q_{21} =1$, or $g_1^2g_2 =1$.  Now apply Propositions~\ref{prop:g1^2g2=1},~\ref{prop:q12=1}, and~\ref{prop:minimal} to yield parts (b) and (c) in the case that $H$ is of type A$_2$ or G$_2$.

For type B$_2$, we have that  \eqref{eq:part(a)Serre} yields the following Serre relation for $H$:
\begin{align*} \label{eq:Serre relation}
x_2^2x_1 - (q_{22}q_{21} + q_{21})x_2 x_1 x_2 + q_{21}^2q_{22}x_1 x_2^2 = 0.
\end{align*}
Arguing as above, with $q_{12}q_{21}q_{22} =1$, we conclude that
$$(1-q_{21}q_{22})(q_{21}-1)(g_1g_2^2-1)=0.$$
So, by \eqref{eq:qij Cartan}, $q_{12} = 1$ or $q_{21} =1$, or $g_1g_2^2 =1$. Apply Propositions~\ref{prop:g1^2g2=1},~\ref{prop:q12=1}, and~\ref{prop:minimal} to yield part (b) and (c) in the case that $H$ is of type B$_2$.
\end{proof}


\section{Galois-theoretical lifts of type A$_1^{\times \theta}$ and of rank two types} \label{sec:noncorad}

In this section, we discuss the Galois-theoretical property of liftings of various finite-dimensional,  pointed, coradically graded Hopf algebras. In particular, we discuss liftings of bosonizations of quantum linear spaces and of Nichols algebras of types A$_2$ and  B$_2$, which are classified by Andruskiewitsch-Schneider \cite{AndSch:p3, AndSch:p4} and Beattie-D{\u{a}}sc{\u{a}}lescu-Raianu \cite{BDR}. To our knowledge, the liftings of Nichols algebras of type G$_2$ have not been classified, so this case is not addressed.

\subsection{Lifting  bosonizations of quantum linear spaces $B(G,g_1,\dots,g_\theta,\chi_1,\dots,\chi_\theta)$, type A$_1^{\times \theta}$}
Recall that $G$ is  a finite abelian group and take $\theta \geq 2$. The finite-dimensional, pointed Hopf algebras $H$ so that $gr(H) \cong B(G,\ug, \uchi)$ (of Section~\ref{subsec:B(G,theta)}) have been classified in \cite[Section~5]{AndSch:p3}. These Hopf algebras, denoted by $A(G, \ug, \uchi, \underline{\alpha}, \underline{\lambda})$, are generated by $G$ and $(g_i, 1)$ skew-primitive elements $x_i$, for $g_i \in G$  with $i = 1, \dots, \theta$. Take character maps $\chi_i \in \widehat{G}$, for $i = 1, \dots, \theta$, so that $q_{ii} := \chi_i(g_i)$ has order $n_i$. Here, we have that 
$$\chi_i(g_j) \chi_j(g_i) = 1.$$
Then $A(G, \ug, \uchi, \underline{\alpha}, \underline{\lambda})$ is subject to the relations of $G$ and 
$$gx_i = \chi_i(g)x_i g, \quad \quad x_i^{n_i} = \alpha_i(1-g_i^{n_i}), \quad \quad
x_ix_j = \chi_j(g_i) x_jx_i + \lambda_{ij}(1-g_ig_j),$$
for $\alpha_i, \lambda_{ij} \in \kk$ and $i \neq j$. We may (and will) assume that 
$\lambda_{ij}=0$ if $g_ig_j=1$.

\begin{theorem} \label{thm:liftQLS}
Let $A'$ be the minimal Hopf subalgebra of $A= A(G, \ug, \uchi, \underline{\alpha}, \underline{\lambda})$ generated  by $\{g_1, \dots, g_{\theta},$ $  x_1, \dots, x_{\theta}\}$. Then, $A$ is Galois-theoretical if and only if $A'$ is the quotient by a group of central grouplike elements of a tensor product of 
\begin{itemize}
\item Hopf algebras $u_q'({\mathfrak{gl}}_2)$ from \cite[Definition~11]{PartI},
\item Taft algebras $T(n,\zeta)$,
\item Nichols Hopf algebras $E(n)$, or 
\item book algebras $\mathbf{h}(\zeta,1)$.
\end{itemize}
\end{theorem}

\begin{proof}
The Hopf algebras $u_q'({\mathfrak{gl}}_2)$, $T(n,\zeta)$, $E(n)$, $\mathbf{h}(\zeta,1)$ are Galois-theoretical by \cite[Propositions~17,~19,~22, and~33]{PartI}, and by \cite[Proposition~10(5)]{PartI}, so is any tensor product $H$ of these Hopf algebras. Also, it is easy to check that any quotient 
of $H$ by a group $Z$ of central grouplike elements 
is Galois-theoretical, by looking at the action of $\overline{H}:=H/(z-1,z\in Z)$ on the field of 
$Z$-invariants $L^Z$ in an inner faithful $H$-module field $L$.  So the ``if" direction follows from Proposition~\ref{prop:minimal}.

Conversely, suppose that $A(G, \ug, \uchi, \underline{\alpha}, \underline{\lambda})$ is Galois-theoretical, then so is the subalgebra $A^i$ generated by $g_i$ and $x_i$ by \cite[Proposition~10(3)]{PartI}. Since $A^i$ is a generalized Taft algebra, then by \cite[Proposition~21]{PartI}, we get that $A^i$ is coradically graded. Thus, $\alpha_i =0$ for all $i$. Now it suffices to show that the minimal Hopf subalgebra $A'$ is a central quotient of  the tensor product of  $u_q'({\mathfrak{gl}}_2)$, $T(n,\zeta)$, $E(n)$, or $\mathbf{h}(\zeta,1)$.

Applying Theorem~\ref{thm:H(G,theta)} and Remark~\ref{rem:w_i} to the relation $x_ix_j = \chi_j(g_i) x_jx_i + \lambda_{ij}(1-g_ig_j)$ yields:
\[
\begin{array}{l}
{w_i}(1-g_i){w_j}(1-g_j) - \chi_j(g_i){w_j}(1-g_j){w_i}(1-g_i) - \lambda_{ij}(1-g_ig_j)\\
=[{w_i}{w_j}(1-\chi_j(g_i))- \lambda_{ij}](1-g_ig_j) =0.
\end{array}
\]
So either ${w_i}{w_j}(1-\chi_j(g_i)) = \lambda_{ij}$ or $g_ig_j=1$ for all $i \neq j$. 

To proceed, define a graph $\Gamma$ whose vertices are labelled $1, \dots ,\theta$ 
and which has two kinds of edges --- dotted ones and solid ones. Namely, 
\begin{itemize}
\item[$\ast$] a solid edge $i$---$j$ if and only if $g_ig_j=1$.
\item[$\ast$] a dotted edge $i \cdots j$ if and only if $g_ig_j\ne 1$ and ${w_i}{w_j}(1-\chi_j(g_i)) = \lambda_{ij}$, with $\lambda_{ij}\ne 0$.
\end{itemize}
Thus, $i$ and $j$ are not connected if and only if $g_ig_j\ne 1$, $\lambda_{ij}=0$, and $\chi_i(g_j)=1$. 
Note that two vertices cannot be connected by a solid and a dotted edge at the same time. 

Note that if $i$---$j$ then the order of $g_i$ equals the order of $g_j$. Also, if $i\cdots j$, then $w_iw_j \in \kk$ commutes with $g_s$  for all~$s$, and $\chi_i(g_s)=\chi_j(g_s)^{-1}$. So $\chi_i(g_i)=\chi_j(g_i)^{-1}=\chi_i(g_j)=\chi_j(g_j)^{-1}$. 
Thus if two vertices are connected by a dotted edge, then $g_i$ and $g_j$ also have the same order.
Hence, to each component $C$ of $\Gamma$ one can attach its order, which is the order of $g_i$ for any $i\in \Gamma$.  This is also the nilpotency order of $x_i$ for any $i\in \Gamma$, since $\alpha_i=0$. 

For a component $C$ of $\Gamma$, let $H(C)$ be the Hopf subalgebra of $A(G, \ug, \uchi, \underline{0}, \underline{\lambda})$ generated by $\{ g_i, x_i \}_{i \in C}$.

Let us consider the case when $\Gamma$ is connected. Assume the order of $\Gamma$ is  $n>2$. 
Suppose we have $i$---$j$---$r$ in $\Gamma$. Then by the proof of Sublemma~\ref{sublem:QLS2}, we come to a contradiction. Suppose now we have $i \cdots j\cdots r$ in $\Gamma$. Then a similar proof leads to a contradiction:
$$
\chi_i(g_j)=\chi_j(g_j)^{-1}=\chi_r(g_j)=\chi_j(g_r)^{-1}=\chi_i(g_r)=\chi_r(g_i)^{-1}=\chi_j(g_i)=\chi_i(g_j)^{-1}.
$$
Finally, suppose that $i\cdots j$---$r$. Then by the above, $i$ is not connected to $r$. So, $\chi_i(g_r)=1$, and $\chi_j(g_r)=\chi_j(g_j)=1$, a contradiction.  

This means that either $\Gamma$ has one 
vertex (so, $H(\Gamma$) = $T(n,\zeta)$) or two vertices connected by a solid edge (so, $H(\Gamma) = \mathbf{h}(\zeta,1)$) or dotted edge (so, $H(\Gamma$) = $u_q'({\mathfrak{gl}}_2)$ or its central quotient; see, e.g. the discussion before \cite[Proposition~30]{PartI}).

Now assume that $\Gamma$ is connected and the order of $\Gamma$ is   $n=2$. Then by the proof of  Sublemma~\ref{sublem:QLS1}, $\Gamma$ is a complete graph with solid and dotted edges. Indeed, the matrix $(\chi_i(g_j))$ is symmetric
(as $\chi_i(g_j)=\pm 1$), thus all its entries for $i,j\in \Gamma$ are the same, so they have to be $-1$.   

It is clear that if $i$---$j$---$r$ then $i$---$r$. Thus, if we only keep the solid edges, 
$\Gamma$ will fall apart into components $C_1,\dots,C_m$, such that for any $i\in C_k$, $j\in C_l$, $k\ne l$, we have $i\cdots j$. Thus if $i,r\in C_s$ and $j\in C_l$, $s\ne l$, then $w_i$ and $w_r$ are both scalar multiples of $w_j^{-1}$, and $g_i=g_r$ (as $n=2$), so we have a contradiction with inner faithfulness. Thus, 
either all edges in $\Gamma$ are solid or all are dotted. If all edges are solid, 
then $H(C)=E(n)$ for some $n$. Also, we cannot have a dotted triangle with vertices $i,j,r$, as 
then ${w_i}$ is proportional to ${w_j}^{-1}$, hence to ${w_r}$, hence to ${w_i}^{-1}$, so ${w_i}\in \kk$, a contradiction. Thus, if all edges of $\Gamma$ are dotted (and there is at least one edge), then $|\Gamma|=2$, 
and we get the Hopf algebra $u_q'({\mathfrak{gl}}_2)$ for $q=-1$.

Finally, let $\Gamma$ be arbitrary. Then, by mimicking Sublemma~\ref{sublem:QLS3}, 
we see that our Hopf algebra $A'$ is a quotient of 
the tensor product of the algebras corresponding to the connected components of $\Gamma$ by a subgroup of central grouplike elements. The proposition is proved. 
\end{proof}

\subsection{Lifting  bosonizations $\mf{B}(V) \# \kk G$, for $V$ of type A$_2$}
Let $H$ be a finite-dimensional, pointed Hopf algebra so that $gr(H) \cong \mf{B}(V) \# \kk G$, for $\mf{B}(V)$ of type A$_2$. Such $H$ have been classified for braiding parameter $q$, a primitive $m$-th root of unity, with $m>1$ an odd integer. Namely,
 by \cite[Theorems~3.6 and~3.7]{AndSch:p4} for $m >3$ and by \cite[Proposition~3.3]{BDR} for $m =3$.
 
 We proceed in the case when $m >3$. In this case, $H$ is generated by $G$ and $x_1, x_2$, where $x_i$ is $(g_i,1)$-skew-primitive and $g_i \in G$ and is subject to the relations of $G$,
\begin{itemize}
\item  $g x_i = \chi_i(g) x_i g$, \hfill for $i = 1,2$,
\item $x_i^m = \mu_i(1-g_i^m)$, \hfill for $i=1,2$,
\item  $\left[x_1x_2 - \chi_2(g_1)x_2x_1\right]x_1 = \chi_1(g_2)x_1 \left[x_1x_2 - \chi_2(g_1)x_2x_1\right]$,
\end{itemize}
along with other relations irrelevant to the proof of the result below. Here,
$\mu_1, \mu_2, \lambda \in \kk$,
$q = \chi_1(g_1) = \chi_2(g_2)$ and $\chi_1(g_2) \chi_2(g_1) = q^{-1}.$

\begin{proposition} \label{prop:liftA2} 
Let $H$ be a finite-dimensional pointed Hopf algebra so that $gr(H) \cong \mf{B}(V) \# \kk G$, for $\mf{B}(V)$ of type A$_2$, as above. If $H$ is Galois-theoretical, then $H$ is coradically graded, in which case we are in the setting of Theorem~\ref{thm:ranktwo}\textnormal{(}b\textnormal{)}. 
\end{proposition}

\begin{proof}
Suppose that $H$ is Galois-theoretical. Then, the subalgebra $H^i$ generated by $g_i$ and $x_i$ is also Galois-theoretical by \cite[Proposition~10(3)]{PartI}, for $i =1,2$. 
Since $H^i$ is a generalized Taft algebra, then by \cite[Proposition~21]{PartI}, $H^i$ is coradically graded. Thus, $\mu_1 = \mu_2 =0$. Now, $H$ is a quotient of $H(G,g_1, g_2,\chi_1,\chi_2)$ studied in Section~\ref{subsec:H(G,theta)}. 

We get by Theorem~\ref{thm:H(G,theta)} and Remark~\ref{rem:w_i} that $x_i$ acts as ${w_i}(1-g_i)$ for $i = 1,2$. From the relations $g_i {w_j} = \chi_j(g_i){w_j}g_i$ and $\chi_1(g_1)\chi_1(g_2)\chi_2(g_1) =1$ and $$\left[x_1x_2 - \chi_2(g_1)x_2x_1\right]x_1 = \chi_1(g_2)x_1 \left[x_1x_2 - \chi_2(g_1)x_2x_1\right],$$ we have that
\[
\begin{array}{l}
q_{21}{w_1}(1-g_1){w_1}(1-g_1){w_2}(1-g_2) 
- (q_{12}q_{21}+1) {w_1}(1-g_1){w_2}(1-g_2){w_1}(1-g_1)\\ 
\quad \quad + q_{12}{w_2}(1-g_2){w_1}(1-g_1){w_1}(1-g_1)\\
= {w_1}^2{w_2}\left[(q_{12}-1)(q_{21}-1)(g_1^2g_2-1) \right]  =0.
\end{array}
\]
So, $q_{21} = \chi_1(g_2) = 1$, $q_{12} = \chi_2(g_1) =1$, or $g_1^2g_2=1$.
By Theorem~\ref{thm:ranktwo}, we have that $H$ is a lift of a finite-dimensional, pointed, coradically graded Hopf algebra of type A$_2$ that is Galois-theoretical.  Now, $H$ is coradically graded by Proposition~\ref{bothcannot}.
\end{proof}


\subsection{Lifting  bosonizations $\mf{B}(V) \# \kk G$, for $V$ of type B$_2$.}
Let $H$ be a finite-dimensional pointed Hopf algebra so that $gr(H) \cong \mf{B}(V) \# \kk G$, for $\mf{B}(V)$ of type B$_2$. Such $H$ have been classified for braiding parameter $q$, a primitive $m$-th root of unity, with $m \neq 5$ an odd integer. Namely,
 by \cite[Theorems~2.6 and~2.7]{BDR}, $H$ is generated by $G$ and $x_1, x_2$, where $x_i$ is $(g_i,1)$-skew-primitive and $g_i \in G$ and is subject to the relations of $G$,
\begin{itemize}
\item $g x_i = \chi_i(g) x_i g$ \hfill for $i = 1,2$,
\item $x_i^m = \mu_i(g_i^m-1)$ \hfill for $i=1,2$;
\item $x_2[x_2x_1 - \chi_1(g_2)x_1x_2] - \chi_1(g_2)\chi_2(g_2)[x_2x_1 - \chi_1(g_2)x_1x_2]x_2=0$,
\end{itemize}
along with other relations irrelevant to the proof of the result below. Here,
$\mu_1, \mu_2, \lambda \in \kk$,
$q^2 = \chi_1(g_1),$ $q = \chi_2(g_2),$ and $\chi_1(g_2) \chi_2(g_1) = q^{-2}.$

\begin{proposition} \label{prop:liftB2} 
Let $H$ be a finite-dimensional, pointed Hopf algebra so that $gr(H) \cong \mf{B}(V) \# \kk G$, for $\mf{B}(V)$ of type B$_2$, as above. If $H$ is Galois-theoretical, then $H$ is coradically graded, in which case we are in the setting of Theorem~\ref{thm:ranktwo}\textnormal{(}b\textnormal{)}.
\end{proposition}

\begin{proof}
Suppose that $H$ is Galois-theoretical. Then, the subalgebra $H^i$ generated by $g_i$ and $x_i$ is also Galois-theoretical by \cite[Proposition~10(3)]{PartI}, for $i =1,2$. 
Since $H^i$ is a generalized Taft algebra, then by \cite[Proposition~21]{PartI}, $H^i$ is coradically graded. Thus, $\mu_1 = \mu_2 =0$ and $m = \text{ord}(g_1) = \text{ord}(g_2)$. 

Now, $H$ is a quotient of $H(G,g_1, g_2,\chi_1,\chi_2)$ studied in Section~\ref{subsec:H(G,theta)}, and we get by Theorem~\ref{thm:H(G,theta)} and Remark~\ref{rem:w_i} that $x_i$ acts as ${w_i}(1-g_i)$ for $i = 1,2$. From the relations $g_i {w_j} = \chi_j(g_i){w_j}g_i$, $\chi_1(g_2)\chi_2(g_1)\chi_2(g_2) =1$ and 
$$x_2\left[x_2x_1 - \chi_1(g_2)x_1x_2\right] - \chi_1(g_2)\chi_2(g_2)\left[x_2x_1 - \chi_1(g_2)x_1x_2\right]x_2=0,$$ we have that
\[
\begin{array}{l}
{w_2}(1-g_2){w_2}(1-g_2){w_1}(1-g_1)
- (q_{21}+q_{21}q_{22}) {w_2}(1-g_2){w_1}(1-g_1){w_2}(1-g_2)\\ 
\quad \quad + q_{21}^2q_{22}{w_1}(1-g_1){w_2}(1-g_2){w_2}(1-g_2) \\
= {w_1}^2{w_2}\left[(q_{12}^{-1}-1)(1-q_{21})(g_1g_2^2-1) \right]  =0.
\end{array}
\]
So, $q_{12} = 1$, $q_{21} =1$, or $g_1g_2^2=1$.
Since $m \neq 5$, we have by Theorem~\ref{thm:ranktwo}  $H$ is a lift or a finite-dimensional, pointed, coradically graded Hopf algebra of type B$_2$ that is Galois-theoretical. Now $H$ is coradically graded by 
Proposition~\ref{bothcannot}. 
\end{proof}


\section{Galois-theoretical small quantum groups of higher rank} \label{sec:smallqgroups} We study the Galois-theoretical property of the small quantum groups $u_q(\mf{g})$, $u_q^{\geq 0}(\mf{g})$, and gr$(u_q(\mf{g}))$, along with their Drinfeld twists. Here, let $\mf{g}$ be a finite-dimensional semisimple Lie algebra of type X. Note that $u_{q}^{\geq 0}(\mf{g})$, $u_q(\mf{g})$, and gr$(u_{q}(\mf{g}))$ (and their twists) are pointed Hopf algebras of finite Cartan type X, X $\times$ X, and  X $\times$ X, respectively. We use the notation and terminology of \cite[Sections~1.2,~1.4, and~3.9]{PartI}.

\begin{proposition} \label{prop:uq(g)}
Let $\mf{g}$ be a finite-dimensional semisimple Lie algebra, not of type A$_1^{\times r}$ for $r \geq 2$. Let $q \in \kk$ be a root of unity of odd order $m \geq 3$, with $m>3$ for $\mf{g}$ containing a component of type G$_2$. Then, we have the statements below.
\begin{enumerate}
\item $u_q(\mf{g})$ is Galois-theoretical if and only if $\mf{g} = \mf{sl}_2$.
\item $u_q^{\geq 0}(\mf{g})$ is Galois-theoretical if and only if $\mf{g} = \mf{sl}_2$. In this case, $u_q^{\geq 0}(\mf{sl}_2)$ is a Taft algebra of dimension $m^2$.
\item gr\textnormal{(}$u_q(\mf{g})$\textnormal{)} is not Galois-theoretical.
\end{enumerate}
\end{proposition}

\begin{proof} 
We have that both $u_q(\mf{sl}_2)$ and $u_q^{\geq 0}(\mf{sl}_2)$ are Galois-theoretical by \cite[Propositions~25(2) and~17]{PartI}. Moreover, gr($u_q(\mf{sl}_2)$) is not Galois-theoretical by \cite[Proposition~25(1)]{PartI}.

On the other hand, let $\mf{g}$ have rank $\geq 2$. Note that there exists $i\neq j$, so that $a_{ij}=-1$ with $|i-j| =1$, since $\mf{g}$ is not of type A$_1^{\times r}$ for $r \geq 2$. The Hopf algebras $u_q(\mf{g})$, $u_{q}^{\geq 0}(\mf{g})$ and gr$(u_{q}(\mf{g}))$ have a Serre relation given as follows:
\begin{align} 
e_i^2 e_j - (q^d +q^{-d})e_ie_je_i +e_je_i^2 = 0 \label{eq:Serre(b)},
\end{align} 
where $i \neq j$ as above, and $e_i$ is a $(k_i, 1)$-skew-primitive element for a grouplike element $k_i$. 
Here, $d=1$ for type ADE, $d=2$ for type BCF, and $d=3$ for type G. 

By way of contradiction, suppose that these Hopf algebras $H$ are Galois-theoretical with module field $L$. These Hopf algebras are quotients of $H(G,\underline{g}, \underline{\chi})$ studied in Section~\ref{subsec:H(G,theta)}, so we get by Theorem~\ref{thm:H(G,theta)} that $e_i$ acts on $L$ by formula~\eqref{eq:formula xi}. By Remark~\ref{rem:w_i}, $e_i$ acts as ${w_i}(1-k_i)$. 
Now the relations between $k_i$ and $e_j$ of $H$ from \cite[Definition~2]{PartI} imply that
$$\chi_i(k_i) = q^{2d}, \quad \quad \chi_i(k_j) = q^{{-d}}, \quad \quad \chi_j(k_i) = q^{-d}, \quad \quad \chi_j(k_j) = q^{2}.$$
As $k_i {w_j} = \chi_j(k_i) {w_j} k_i$, we have from \eqref{eq:Serre(b)} that
\[
\begin{array}{rl}

0 =& {w_i}(1-k_i){w_i}(1-k_i){w_j}(1-k_j)
-(q^d +q ^{-d}){w_i}(1-k_i){w_j}(1-k_j){w_i}(1-k_i)\\
\medskip

&\quad \quad+{w_j}(1-k_j){w_i}(1-k_i){w_i}(1-k_i)\\
=&{w_i}^2 {w_j}[
(1-q^d k_i)(1-q^{-d}k_i)(1-k_j)
-(q^d+q^{-d})(1-q^dk_i)(1-q^{{-d}}k_j)(1-k_i)\\
\medskip

&\quad \quad +(1-q^{-2d}k_j)(1-q^{2d}k_i)(1-k_i)]\\

=&
\begin{cases}
{w_i}^2 {w_j}[q^{-1}(q-1)^2(k_i^2k_j-1)] & \text{ if } d=1, \\
{w_i}^2 {w_j}[q^{-2}(q-1)^2(q+1)^2(k_i^2k_j-1)] & \text{ if } d=2, \\
{w_i}^2 {w_j}[q^{-3}(q-1)^2(q^2+q+1)^2(k_i^2k_j-1)] & \text{ if } d=3.
\end{cases}
\end{array}
\]
Given the conditions on the order of $q$ and the fact that $k_i^2k_j \neq 1$, we arrive at a contradiction. Thus, $u_q(\mf{g})$, $u_q^{\geq 0}(\mf{g})$, and gr($u_q(\mf{g})$) are not Galois-theoretical when $\mathfrak{g}$ is of rank $\geq 2$ not of type A$_1^{\times r}$. 
\end{proof}

The following result characterizes the Galois-theoretical properties of Drinfeld twists of small quantum groups.

\begin{proposition} \label{prop:uq(g)twist} Let $\mf{g}$ be a finite-dimensional simple Lie algebra, and retain the notation of Proposition~\ref{prop:uq(g)}. Let $m$ be relatively prime to det$(a_{ij})$, and to 3 in type G$_2$. Let $J$ be a Drinfeld twist of $u_q(\mf{g})$ induced from its Cartan subgroup. Then, we have the following statements.
\begin{enumerate}
\item There are precisely $2^{\text{rank}(\mf{g})-1}$ twists $J$ \textnormal{(}up to gauge transformation\textnormal{)} so that $u_q^{\geq 0}(\mf{g})^J$ is Galois-theoretical. 
\item The Hopf algebra $u_q(\mf{g})^J$ can be Galois-theoretical if and only if $\mf{g} = \mf{sl}_n$. In this case, there are only two of such twists $J$
for $n \geq 3$, and one \textnormal{(}namely, $J=1$\textnormal{)} for $n=2$, up to a gauge transformation.
\end{enumerate}

\end{proposition}

\begin{proof}
(a) \cite[Proposition~37]{PartI} provides $2^{\text{rank}(\mf{g})-1}$ twists giving rise to Galois-theoretical examples, so our job is to show that no other 
twist works. Let $J$ be a twist such that $u_q^{\geq 0}(\mf{g})^J$ is Galois-theoretical. Then by the classification in the rank $2$ 
case (Theorem \ref{thm:ranktwo}), for each $i,j$, either $q_{ij}=1$ or $q_{ji}=1$ (and both hold if $i,j$ are not connected). This 
defines an orientation on the Dynkin diagram of $\mf{g}$ (with $i\to j$ if $q_{ij}\ne 1$), and there are exactly $2^{\text{rank}(\mf{g})-1}$ possibilities. 
Thus, there are no possibilities for $J$ beyond the above, proving~(a). 

(b)
Let $J$ be a twist such that $u_q(\mathfrak{g})^J$ be Galois-theoretical, and $b_J$ be the alternating bicharacter corresponding to $J$ as in \cite[Proposition~7]{PartI}.
Let us consider generators $e_i$ and $f_j$ of $u_q(\mathfrak{g})^J$, for $j\ne i$. Together with Cartan elements, they generate a coradically graded Hopf subalgebra 
of type $A_1\times A_1$. Let $H_{ij}$ be its minimal Hopf subalgebra. Then by Theorem \ref{thm:ranktwo}, $H_{ij}$ is either a tensor product of two Taft algebras or a book algebra of 
type $\mathbf{h}(\zeta,1)$. 

Let $b_{ij}:=b_J(k_i,k_j)$. It is easy to check that the Cartan matrix of $H_{ij}$ has the form 
$$
A_{ij}=\biggl(\begin{matrix} q^{2d_i} & q^{-d_ia_{ij}}b_{ij}\\ q^{d_ja_{ji}}b_{ij}^{-1} & q^{-2d_j}\end{matrix}\biggr) 
$$
So, if $H_{ij}$ is a tensor product of two Taft algebras, we must have $b_{ij}=q^{d_ia_{ij}}$. Then $b_{ji}\neq q^{d_ja_{ji}}$, as the matrix $(d_ia_{ij})$ is symmetric, while 
$b_{ij}=b_{ji}^{-1}$. Thus, if $i$---$j$ then $H_{ij}$ or $H_{ji}$ is a book algebra. 

If $H_{ij}$ is a book algebra, then 
\begin{equation}\label{bookal}  
q^{2d_i}=q^{-d_ia_{ij}}b_{ij},
\end{equation} 
 and if $H_{ji}$ is a book algebra, then 
\begin{equation}\label{bookal1}  
q^{-2d_i}=q^{d_ia_{ij}}b_{ji}^{-1}. 
\end{equation} 
Taking the product of \eqref{bookal},\eqref{bookal1}, we get $b_{ij}=b_{ji}$, a contradiction. Thus, $H_{ij}$ and $H_{ji}$ cannot both be book algebras.
So, if $i$---$j$, then exactly one of the algebras $H_{ij},H_{ji}$ is a tensor product and exactly one is a book algebra. 
 
Introduce an orientation on the Dynkin diagram of $\mf{g}$ by putting $i\to j$ if $H_{ij}$ is a book algebra. 
If $i\to j$, then $H_{ji}$ is a tensor product of Taft algebras, so 
$b_{ji}=q^{d_ja_{ji}}$ and hence $b_{ij}=q^{-d_ia_{ij}}$. 
Thus, by \eqref{bookal} we have $q^{2d_i}=q^{-2d_ia_{ij}}$. Also $q^{2d_i}=q^{2d_j}$. So $a_{ij}=a_{ji}=-1$, 
i.e., $\mf{g}$ is simply laced.

Next, 
by (the proof of) Sublemma~\ref{sublem:QLS2}, 
we cannot have $i\to j$ and $r\to j$ for $i,j,r$ distinct, 
or $j\to i$ and $j\to r$ for $i,j,r$ distinct. 
Hence, the Dynkin diagram of $\mf{g}$ cannot contain a triple vertex. 
So the diagram $Q$ and the Lie algebra $\mf{g}$ 
must be of type $A_{n-1}$ for some $n$. Moreover, the above constraint on the orientation also implies that all the edges of the Dynkin diagram
are oriented in the same direction, i.e. we are left with just two orientations (which coincide if $n=2$, i.e., there are no edges). 

Finally, the twists $J^\pm$ corresponding to the remaining two orientations indeed give rise to Galois-theoretical Hopf algebras 
by \cite[Corollary~31]{PartI}. The theorem is proved. 
\end{proof}


\section{Further directions} \label{sec:directions}

In this work we studied the Galois-theoretical property of the most extensively studied class of finite-dimensional, pointed Hopf algebras: those of finite Cartan type.  The study of the Galois-theoretical property of  such Hopf algebras of higher rank ($\geq 3$) is open in general (and some cases of liftings in rank 2 have also not been treated here). We also suggest the following tasks: investigate the Galois-theoretical property of finite-dimensional, pointed Hopf algebras $H$ with $G(H)$ abelian 
\begin{itemize}
\item of {\it standard type} (which properly includes finite Cartan type) \cite{AndAngiono, Angiono:standard},
\item of {\it super type} \cite{AAY}, or
\item of {\it unidentified type} \cite{Angiono:unidentified}.
\end{itemize}
Moreover, it would be interesting to continue this study for known finite-dimensional, pointed Hopf algebras over a non-abelian group of grouplike elements \cite{AG, Zoo}, or over a cosemisimple Hopf algebra (especially when the Nichols algebra is of Cartan type A) \cite{AAI}. 


\section*{Acknowledgments}
The authors are grateful to Nicol\'{a}s Andruskiewitsch and Iv\'{a}n Angiono for supplying numerous suggestions and valuable references.
The authors were supported by the National Science Foundation: NSF-grants DMS-1502244 and DMS-1550306.

\def\cprime{$'$}


\end{document}